 \documentclass[reqno,11pt]{amsart}
 \usepackage{amsmath, amsthm, a4, latexsym, amssymb}
\usepackage[unicode]{hyperref}
\usepackage{srcltx}
\usepackage{xcolor}

\def\@rmrk#1#2{\refstepcounter
    {#1}\@ifnextchar[{\@yrmrk{#1}{#2}}{\@xrmrk{#1}{#2}}}

\makeatletter\@addtoreset{equation}{section}\makeatother

 \newfont{\bfit}{cmbxti10 scaled 1200}

\renewcommand{\d}{{\rm d}}

 \newcommand{\e}{{\rm e} }

 \newcommand{\eps}{\varepsilon}

 \newcommand{\R}{\mathbb{R}}
 \newcommand{\N}{\mathbb{N}}
 \newcommand{\Z}{\mathbb{Z}}

 \newcommand{\E}{\mathbb{E}}
 \renewcommand{\P}{\mathbb{P}}
 \def\1{{\mathchoice {1\mskip-4mu\mathrm l} 
{1\mskip-4mu\mathrm l}
{1\mskip-4.5mu\mathrm l} {1\mskip-5mu\mathrm l}}}

 \newcommand{\Mcal}{{\mathcal M}}

\newcommand{\ignore}[1]{}
\newcommand{\heap}[2]{\genfrac{}{}{0pt}{}{#1}{#2}}

\newcommand{\ssup}[1] {{\scriptscriptstyle{({#1}})}}

\renewcommand{\subsection}{\secdef \subsct\sbsect}
\newcommand{\subsct}[2][default]{\refstepcounter{subsection}
\vspace{0.15cm}
{\flushleft\bf \arabic{section}.\arabic{subsection}~\bf #1  }
\nopagebreak\nopagebreak}
\newcommand{\sbsect}[1]{\vspace{0.1cm}\noindent
{\bf #1}\vspace{0.1cm}}

{\nopagebreak {\hfill\rule{2mm}{2mm}}\\ }

\newtheorem{theorem}{Theorem}[section]
\newtheorem{lemma}[theorem]{Lemma}
\newtheorem{cor}[theorem]{Corollary}

\newtheoremstyle{thm}{1.5ex}{1.5ex}{\itshape\rmfamily}{}
{\bfseries\rmfamily}{}{2ex}{}

\newtheoremstyle{rem}{1.3ex}{1.3ex}{\rmfamily}{}
{\itshape\rmfamily}{}{1.5ex}{}
\theoremstyle{rem}
\newtheorem{remark}{{\slshape\sffamily Remark}}[]

\refstepcounter{subsubsection}

\def\thebibliography#1{\section*{References}
  \list%
  {\arabic{enumi}.}
    {\settowidth\labelwidth{[#1]}\leftmargin\labelwidth
    \advance\leftmargin\labelsep
    \parsep0pt\itemsep0pt
    \usecounter{enumi}}
    \def\newblock{\hskip .11em plus .33em minus .07em}
    \sloppy                   
    \sfcode`\.=1000\relax}

\begin{document}
\title[Large deviations for SRW on percolation clusters]
{\large Quenched large deviations for simple random walks on percolation clusters including long-range correlations}
\author[Noam Berger, Chiranjib Mukherjee and Kazuki Okamura]{}
\maketitle
\thispagestyle{empty}
\vspace{-0.5cm}

\centerline{\sc Noam Berger \footnote{Hebrew University Jerusalem and TU Munich, Boltzmannstrasse 3, Garching 85748, {\tt noam.berger@tum.de}},
Chiranjib Mukherjee \footnote{University of M\"unster, Einsteinstrasse 62, M\"unster 48149, Germany, {\tt mukherjee@cims.nyu.edu}} and 
Kazuki Okamura\footnote{Research Institute for Mathematical Sciences, Kyoto University, Kyoto, 606-8502, {\tt kazukio@kurims.kyoto-u.ac.jp} }}
\renewcommand{\thefootnote}{}

\vspace{-0.5cm}
\centerline{\textit{TU Munich and Hebrew University Jerusalem, University of M\"unster, Kyoto University}}
\vspace{0.2cm}


\begin{center}
\today
\end{center}

\begin{quote}{\small {\bf Abstract:} 
We prove a {\it{quenched}} large deviation principle (LDP) for a simple random walk on a supercritical percolation cluster (SRWPC) on $\Z^d$ ($d\geq 2$). 
The models under interest include classical Bernoulli bond and site percolation as well as models that exhibit long range correlations, like the random cluster model, the random interlacement and {the vacant set of random interlacements}(for $d\geq 3$) and the level sets of the Gaussian free field ($d\geq 3$). 

Inspired by the methods developed by Kosygina, Rezakhanlou and Varadhan (\cite{KRV06}) for proving quenched LDP for elliptic diffusions with a random drift,
and by Yilmaz (\cite{Y08}) and Rosenbluth (\cite{R06}) for similar results regarding elliptic random walks in random environment,
we take the point of view of the moving particle and prove a large deviation principle for the quenched distribution of the {\it{pair empirical measures}} of the environment Markov chain
in the non-elliptic case of SRWPC . Via a contraction principle, this reduces easily to a quenched LDP for the distribution of the mean velocity of the 
random walk and both rate functions admit explicit variational formulas. 
 
The main difficulty in our set up lies in the inherent non-ellipticity as well as the lack of {\it{translation-invariance}} stemming from conditioning on the fact that the origin belongs to the infinite cluster. 
We develop a unifying approach for proving quenched large deviations for SRWPC based on exploiting coercivity properties of the relative entropies in the context of convex variational analysis, combined with input from ergodic theory and invoking geometric properties 
of the supercritical percolation cluster. 
}


 

\end{quote}
\section{Motivation, introduction and main results}


We consider a simple random walk on the infinite cluster of some bond and site percolation models on $\Z^d$, $d\geq 2$. 
The percolation models under interest include classical Bernoulli bond and site percolation, as well as 
models that exhibit long-range correlations, including {the random-cluster model}, 
random interlacements and the vacant set of random interlacements in $d\geq 3$, and the level set of the Gaussian free field (also for $d\geq 3$).
Conditional on the event that the origin lies in the infinite open cluster, it is known that a law of large numbers 
and quenched central limit theorem hold (see \cite{SS04}, \cite{MP07}, \cite{BB07} and \cite{PRS15}). 
Treatment of these classical questions for these models need care because of its inherent {\it{non-ellipticity}} -- a problem which permeates in several forms in the above mentioned literature.

Questions on large deviation principles (LDP) in the quenched setting for general random walks in {\it{elliptic}} random environments (RWRE) have also been studied. In $d=1$, first Greven and den Hollander (\cite{GdH94}) for  
i.i d. and uniformly elliptic random environments, and then Comets, Gantert and Zeitouni (\cite{CGZ00}) for stationary, ergodic and uniformly elliptic random environments, derived quenched LDP for the mean velocity of a a RWRE
and obtained explicit variational formulas for the rate function. For $d\geq 1$, Zerner (\cite{Z98}, see also Sznitman (\cite{S94})) proved quenched LDP under the assumption that the logarithm of the random walk transition probabilities possesses finite $d$-th moment and the random environment enjoys the {\it{nestling property}}. 
His method is based on proving shape theorems invoking the sub-additive ergodic theorem. 
Using the sub-additivity more directly, Varadhan (\cite{V03}) proved a quenched LDP dropping the {\it{nestling}} assumption and assuming uniform ellipticity for the random environment. 
However, the use of sub-additivity in the above results did not lead to any desired formula for the rate function.

Kosygina, Rezakhanlou and Varadhan (\cite{KRV06}) derived a novel method for proving quenched LDP using the {\it{environment seen from the particle}} in the context
of a diffusion with a random drift assuming some growth conditions on the random drift (ellipticity) and obtained a variational formula for the rate function.
This method goes parallel to quenched homogenization of random Hamilton-Jacobi-Bellman (HJB) equations.
Rosenbluth (\cite{R06}) adapted this theory to the ``level-1" large deviation analysis of the rescaled location of a multidimensional random walk in random environments and also obtained a formula for the rate function.
The assumption regarding the growth condition on the random drift imposed in \cite{KRV06} under which homogenization of HJB takes place, or quenched large deviation principle for the rescaled law of the diffusion
holds, now translates to the assumption that logarithm of the random walk transition probabilities possesses finite $d+\eps$ moment, for some $\eps>0$ (see \cite{R06}). 
Under the same moment assumption, Yilmaz (\cite{Y08}) extended this work to a ``level-2"  LDP for the law of the pair empirical measures of the environment Markov chain and subsequently Rassoul-Agha and Sepp\"al\"ainen (\cite{RS11}) proved a ``level-3" LDP for the empirical process for the environment Markov chain. 
Like Rosenbluth (\cite{R06}), both \cite{Y08} and \cite{RS11} obtained variational formulas for the corresponding
rate functions. This method has been further exploited for studying free energy for directed and non-directed random walks in a unbounded random potential (see the works of Rassoul-Agha, Sepp\"al\"ainen and Yilmaz \cite{RSY13, RSY14} and Georgiou et al. \cite{GRSY13,GRS14}). We also refer to the works of Armstrong and Souganidis (\cite{AS12}, see also \cite{LS05, LS10, AT14}) for the continuous analogue of \cite{RSY13} concerning homogenization of random Hamilton Jacobi Bellman equations in unbounded environments. Roughly speaking, all these  results in the aforementioned literature work only under the assumption that $V:=-\log \pi \in L^p(\P)$ with $p>d$, where
$\pi$ denotes the random walk transition probabilities in the elliptic random environment whose law is denoted by $\P$.
Thus, the aforementioned  literature does not cover the case $V= \infty$ pertinent to the case of a random walk on a supercritical percolation cluster, an important model that carries 
the aforementioned inherent non-ellipticity of the random environment.


In this context, it is the goal of the present article to develop a unifying approach for proving quenched large deviation principles for the distribution of the empirical measures of the environment Markov chain of SRWPC ({\it{level-2}}) and subsequently deduce the particle dynamics of the rescaled location ({\it{level-1}}) of the walk on the cluster. 
We start with a precise mathematical  layout of  the random environments under consideration including the bond and site percolations on $\Z^d$.




\subsection{The percolation models under interest.}\label{sec-intro-models}

{We fix $d\geq 2$ and denote by $\mathbb B_d$ the set of nearest neighbor edges of the lattice $\Z^d$ and by $\mathcal U_d=\{\pm e_i\}_{i=1}^d$ the set of edges from the origin to its nearest neighbor. 
We will now phrase out the basic set up of the bond and site percolation models on $\mathbb{Z}^d$ which we will be of  working with in this article.}

{For every bond percolation model, 
 we will set $\Omega = \{0,1\}^{\mathbb B_d}$ to be the space of all {\it{percolation configurations}} $\omega=(\omega_b)_{b\in\mathbb B_d}$. 
In other words, 
$\omega_b=1$ refers to the edge $b$ being {\it{present}} or {\it{open}}, 
while $\omega_b=0$ 
implies that it is {\it{vacant}} or {\it{closed}}. 
We consider a random subgraph $\mathcal{C}$ of $\mathbb{Z}^d$ whose vertices and edges are $\mathbb{Z}^d$ and the set of open edges, respectively.  
We call each connected component of the random graph a {\it{cluster}}, and if a cluster contains infinitely many vertices then we call it an {\it infinite cluster}.
For $x, y \in \mathbb{Z}^d$, 
we write $x \sim y$ if $x$ and $y$ are connected in $\mathcal{C}$.   
If $x \sim y$, let $\d_{\textrm{ch}}(x,y)$ be the graph distance on the cluster containing $x$ and $y$, specifically, the minimal length of paths connecting $x$ and $y$ in $\mathcal{C}$.}
Let $\mathcal B$ be the Borel-$\sigma$-algebra on $\Omega$ defined by the product topology. 
We call elements of $\mathcal B$ {\it events} and  say that an event $A$ is {\it increasing} if the following holds: 
Whenever $\omega = (\omega_b)_b \in A$ and $\omega^{\prime} = (\omega^{\prime}_b)_b $ satisfies that $\omega^{\prime}_b \ge \omega_b$ for each $b \in \mathbb B_d$, 
$\omega^{\prime} \in A$.   
Note that $\Z^d$ acts as a group on $(\Omega, \mathcal B)$ 
via translations. 
In other words, 
for each $x\in \Z^d$, $\tau_x: \Omega \longrightarrow \Omega$ acts as a {\it{shift}} given by $(\tau_x\omega)_b= \omega_{x+b}, b \in \mathbb B_d$.
Let $\P$ be a probability measure on $\Omega$. 
Let $\Omega_0 = \{\omega\colon 0\in \mathcal C_\infty(\omega)\}$, and 
if $\P(\Omega_0)  > 0$, 
we then define the conditional probability $\P_0$ by
$$
\P_0(A) = \P\big(A\big|\Omega_0\big) \qquad A\in\mathcal B.
$$

{If we consider a site percolation model, 
then we let $\Omega = \{0,1\}^{\mathbb Z^d}$.
We agree to call a site $x\in \Z^d$ {\it{present}} or {\it{open}} if $\omega_x =1$, 
and {\it{vacant}} or {\it{closed}} if $\omega_x =0$ .
The notation we set up for the bond percolation model in the above paragraph now carry over 
to the site percolation set up  pertaining to a random subgraph $\mathcal{C}$ of $\mathbb{Z}^d$ whose vertices and edges are the set of open sites and the set of edges whose two endpoints are open sites, respectively.}

{We will now postulate a set of conditions imposed on the bond and site percolation models 
and subsequently describe the explicit models under interest that satisfy these conditions.
The general requirements are the following: \\ 
{\bf{Assumption 1.}} For $\P$-a.e. $\omega$, there exists a unique infinite cluster $\mathcal C_\infty(\omega)$ in $\Z^d$. Note that under this assumption, $\P(\Omega_0)  > 0$ and consequently, $\P_0$ is well-defined. \\ 
{\bf{Assumption 2.}} For each $x \in \mathbb{Z}^d \setminus \{0\}$,  $\P$ is invariant and ergodic with respect to the transformation $\tau_x$. \\
{\bf{Assumption 3.}}  We assume that there exist  $c_1, \dots, c_4 > 0$ such that for each $x \in \mathbb{Z}^d$, 
\[ \P[\d_{\textrm{ch}}(0,x) \ge c_1 |x|_1 ; 0,x \in {\mathcal C}_\infty(\omega) ] \leq c_2 \exp(-c_3 (\log |x|_1)^{1+c_4}). \] 
\vspace{2mm}

We will need to impose further assumptions. Let us first define, for any fixed $\omega\in\Omega_0$ and $e\in\mathcal U_d$, 
\begin{equation}\label{def-n}
k(\omega,e)=\inf\{k\geq 1: \tau_{ke}\,\,\omega\in\Omega_0\}.
\end{equation}
Note that under Assumption 2, by the Poincar\'e recurrence theorem (cf. \cite[Section 2.3]{P89}), 
$k(\omega,e)$ is finite $\P_0$-a.s.

{\bf{Assumption 4.}} With the above definition of $k(\omega,e)$, we
then, assume that there exist $c_5, c_6>0$ so that
\[ \P_0 \big[d_{\mathrm{ch}}(0, k(\omega, e)e) > n \big] \le c_5 \exp(-c_6 n). \]
{\bf{Assumption 5.}} The FKG inequality holds, specifically, $\P(A \cap B) \ge \P(A)\P(B)$ 
holds for every two increasing events $A$ and $B$ in $\Omega$.
}

We now turn to a precise description of the specific models that we will be concerned with, and all the following models will satisfy our requirements listed above (see Lemma \ref{chemdist} - Lemma \ref{lemma-FKG}).

\subsubsection{The bond percolation models.} \label{sec-intro-bond}
We first describe two classical models related to bond percolation.



\noindent $\bullet$ {\bf{I.I.D. bond percolation.}}
We fix the {\it{percolation parameter}} $p\in(0,1)$ and denote by 
$$
\P=\P_p:=\big(p \delta_1+ (1-p) \delta_0\big)^{\mathbb B_d}
$$ 
the product measure with marginals $\P(\omega_b=1)=p=1- \P(\omega_b=0)$.
Note that the product measure $\P$ is invariant under the action of the translation group $\{\tau_x\}_x$. 
It is known that there is a critical percolation  probability $p_c=p_c(d)$ which is the infimum of all $p$'s such that $\P(0\in \mathcal C_\infty)>0$. 
In this paper we only consider the case $p>p_c$. 
By Burton-Keane's uniqueness theorem (\cite{BK89}), the infinite cluster is unique and so $\mathcal C_\infty$ is connected with $\P$-probability one. 



\medskip

\noindent $\bullet$ {\bf{Random cluster model.}} The second example is the random-cluster model, which is a natural extension of Bernoulli bond percolation. However, 
this models exhibits long range correlations and one necessarily drops the i.i.d. structure present in the first example. Let us shortly recall
the basic structure and the salient properties of this model.

Let $d \ge 2$, $p \in [0,1]$, $q \ge 1$, and let also
$\Lambda$ be a box in $\mathbb{Z}^{d}$ with boundary condition $\xi \in \{0,1\}^{\mathbb B_d}$. 
Let $\mathbb{P}_{\Lambda, p, q}^{\xi}$ be the random-cluster measure
on $\Lambda$, defined as
$$ 
\mathbb{P}_{\Lambda, p, q}^{\xi}(\{\omega\}) = \frac{1}{Z} \,\, p^{n(\omega)}\,\, (1-p)^{|\Lambda| - n(\omega)}\,\, q^{o(\omega)}. 
$$
Here $Z$ is a normalizing constant that makes $\mathbb{P}_{\Lambda, p, q}^{\xi}$ a probability measure, while $n(\omega)$ is the number of edges in $\Lambda \cap \omega$, $|\Lambda|$ is the number of all edges in $\Lambda$. 
$o(\omega)$ is the number of open clusters of $\omega_{\Lambda, \xi}$ intersecting $\Lambda$, 
where 
$$
\omega_{\Lambda, \xi} =\begin{cases}
\omega \quad\mbox{on}\,\,\Lambda
\\
\xi \quad\mbox{outside}\,\, \Lambda.
\end{cases}
$$
Let 
$$
\mathbb{P}_{p,q}^{\ssup{ b}} = \,\,\, \lim_{\Lambda \to \mathbb{Z}^d} \mathbb{P}_{\Lambda, p, q}^{{\ssup{b}}}
$$
In other words, $\mathbb{P}_{p,q}^{\ssup{b}}$ is the extremal infinite-volume limit random-cluster measures, with free (for $b=0$) and wired (for $b=1$) conditions respectively. 
For each $b\in \{0,1\}$, let
$$
p_{c}^{\ssup b}(q) = \inf \bigg\{p \in [0,1] : \mathbb{P}_{p,q}^{\ssup b}(0 \leftrightarrow \infty) > 0\bigg\}, \quad b  = 0, 1.
$$
Then, 
$p_{c}^{\ssup 0}(q) = p_{c}^{\ssup 1}(q) \in (0,1)$ (\cite[(5.4)]{G06}) and we write this as $p_{c}(q)$. 
It is well-known that, for both $b=0$ and $b=1$, 
the measure $\P:=\mathbb{P}_{p, q}^{\ssup{ b}}$
is invariant and ergodic with respect to $\tau_{x}$ for every $x \in \Z^{d} \setminus \{0\}$ and for all $p \in [0,1]$ and $q\geq 1$ (\cite[(4.19) and (4.23)]{G06}).
Furthermore, for every $p > p_{c}(q)$, there exists a unique infinite cluster $\mathcal{C}_{\infty}$, $\mathbb{P}_{p,q}^{b}$-a.s. by \cite[Theorem 5.99]{G06},

For our purpose, we also need the notion of {\it{slab critical probability}}, which is defined as follows.
For $d \ge 3$, 
we let 
\begin{equation}\label{slab-d3}
\begin{aligned}
&S(L, n) := [0, L-1] \times [-n, n]^{d-1} \\
& \widehat{p}_{c}(q, L) := \inf \left\{ p : \liminf_{n \to \infty} \inf_{x \in S(L,n)} \mathbb{P}_{S(L,n), p,q}^{\ssup{ 0}}(0 \leftrightarrow x) > 0 \right\} \\
&\widehat{p}_{c}(q) := \lim_{L \to \infty} \widehat{p}_{c}(q, L).
\end{aligned}
\end{equation}
For $d = 2$, for $e_{n} = (n, 0) \in \mathbb{R}^{2}$, we let
\begin{equation}\label{slab-d2}
\begin{aligned}
& {p_{g}(q)} := \sup \left\{p : \lim_{n \to \infty} \frac{- \log \mathbb{P}_{p,q}^{\ssup{ 0}}(0 \leftrightarrow e_{n})}{n} > 0 \right\}, \\
&\widehat{p}_{c}(q) := \frac{q(1-p_{g}(q))}{p_{g}(q) + q(1 - p_{g}(q))}
\end{aligned}
\end{equation}
 and we have the bound ${1 > } \ \widehat{p}_{c}(q) \ge p_{c}(q)$. 
Although equality is believed to be true in the last relation (\cite[Conjecture 5.103]{G06}), 
to the best of our knowledge, 
the only known proofs are available only for the case $q = 1$ (i.e., the case of Bernoulli bond percolation (see Grimmett and Marstrand \cite{GM90}),
and for $d = 2$ and every $q\geq 1$ (see Beffara and Duminil-Copin \cite{BD12}), 
and for $d \ge 3$ and $q=2$ (i.e., {\it{FK-Ising model}}, see Bodineau \cite{B05}). 
We will henceforth work in the regime that 
$$ 
p > \widehat{p}_{c}(q),
$$ 
and will write $\P= \mathbb{P}_{p, q}^{\ssup{ b}}$ and $\P_0=\P(\cdot|\,0\in \mathcal C_\infty)$ throughout the rest of the article.
If $d\geq 3$, in order to show that Assumption 4 is satisfied by the random cluster model, for technical reasons we 
will consider only the free boundary case.

We point out that in the process of proving our main results (stated in Section \ref{sec-results}) corresponding to the random cluster model, 
we prove some geometric properties of this model as a necessary by-product. 
In particular, we prove a ``chemical distance estimate" between two points in the infinite cluster $\mathcal C_\infty$
(see {Lemma \ref{chemdist}}), and also obtain exponential tail bounds for the graph distance between the origin and the ``first arrival" of the infinite cluster $\mathcal C_\infty$ on each coordinate direction (see {Lemma \ref{lemma-ell}}). 
Although both results are part of the standard folklore in the i.i.d. percolation literature, the proofs of these two assertions for the random cluster model seem to be new, to the best of our knowledge.



\subsubsection{Site percolation models.}\label{sec-intro-site} \ 
The second class of models we are interested in concerns {\it{site percolations}}, which include the classical Bernoulli i.i.d. percolation 
as well as models that carry long-range correlation. 
We turn to 
short descriptions of these models. 

\noindent $\bullet$ {\bf{Random interlacements in $d\geq 3$.}}  
This model was introduced by Sznitman \cite{Sz10}. 
Let $\mathbb T_N= \big(\Z/N\Z\big)^d$ be the discrete torus in $d\geq 3$. 
For every $u>0$, the {\it{random interlacement}} $\mathcal I^{\ssup u}$ is defined to be a subset of $\Z^d$ 
which arises as the local limit, as $N\to\infty$ of the sites visited by a simple random walk in $\mathbb T_N$
until time $\lfloor uN^d \rfloor$. 
For every finite subset $K\subset\Z^d$ with capacity $\mathrm{cap}(K)$, the distribution of $\mathcal I^{\ssup u}$ is given by
$$
\P\big[\mathcal I^{\ssup u} \cap K= \emptyset\big]= \e^{-u \, \mathrm{cap}(K)},
$$
Furthermore, {for every $u > 0$,} 
$\P$-almost surely, the set $\mathcal I^{\ssup u}$ is an infinite connected subset of $\Z^d$ (see \cite[(2.21)]{Sz10}), exhibits 
long range correlations given by 
\begin{equation}\label{RI-corr}
\bigg|\P\big[x,y\in \mathcal I^{\ssup u}\big]- \P\big[x\in \mathcal I^{\ssup u}\big] \, \P\big[y\in \mathcal I^{\ssup u}\big]\bigg| \sim \big(1+|x-y|\big)^{2-d}.
\end{equation}
{See \cite[(1.68)]{Sz10} for details.} 
\medskip 

\noindent {$\bullet$ {\bf{Vacant set of the random interlacements in $d\geq 3$.}}}  
The {\it{vacant set of random interlacements}} $\mathcal V^{\ssup u}$ is defined to be the complement of the random interlacement $\mathcal I^{\ssup u}$ at level $u$, i.e., 
$$
\mathcal V^{\ssup u} = \Z^d \setminus \mathcal I^{\ssup u} \qquad \P\big[K\subset \mathcal V^{\ssup u} \big]= \e^{-u \mathrm{cap}(K)}.
$$
Furthermore, $\mathcal V^{\ssup u}$ also exhibits polynomially decaying correlation as in \eqref{RI-corr}. It is known that 
there exists $u_\star\in (0,\infty)$ such that almost surely, for every $u > u_\star$, all connected components of $\mathcal V^{\ssup u}$ are finite (\cite{TW11}),
while for $u < u_\star$,  $\mathcal V^{\ssup u}$ contains an infinite connected component $\mathcal C_\infty$,
which is unique (\cite{T09aap}).

Geometric properties of the random interlacements and {the vacant sets of them} have been studied extensively. 
{Cerny-Popov (\cite{CP12}) obtained sharp estimates on the graph distance in random interlacement}, and,  
Drewitz-Rath-Sapozhnikov (\cite{DRS14}) obtained sharp estimates on the graph distance { in the vacant set of random interlacement}, assuming that $u\in (0,\overline u)$ for some
$\overline u \leq u_\star$, where $\overline u$ is introduced in \cite[Theorem 2.5]{DRS14}.
Although it is believed that $\overline u= u_\star$, we will henceforth assume that 
$$
u<\overline u.
$$
and in this regime, as before, we will write $\P_0=\P(\cdot|\, 0\in \mathcal C_\infty)$.

\medskip

\noindent $\bullet$ {\bf{{Level sets of } Gaussian free fields in $d\geq 3$.}} 
This model has a strong background in statistical physics (see \cite{LS86} and \cite{Sh07} for a mathematical survey).    
The Gaussian free field on $\Z^d$ for $d \geq 3$, is a centered Gaussian field $\varphi = \big(\varphi(x)\big)_{x\in Z^d}$ 
under the probability measure $\P$ with covariance function
$$
\E[\varphi(x)\varphi(y)] = g(x,y)=c_d |x-y|^{2-d},
$$
given by the Green function of the simple random walk on 
$\Z^d$. This leads to long range correlations exhibited by random field $\varphi$.
For every $h\in \R$, the {\it{excursion set above level $h$}} is defined as
$$
E^{\geq h} = \{x \in \Z^d \colon \varphi(x)\geq h\}
$$
and it is known that there exists $h_\star\in [0,\infty)$ such that for every 
$h<h_\star$, $\P$-almost surely, $ E^{\geq h}$ contains a unique infinite connected component
and for every $h > h_\star$, all the connected components of $ E^{\geq h}$
are finite. Like in the case of random interlacements and vacant set of random interlacements,
results on the graph distance for the excursion level set $ E^{\geq h}$ were also obtained 
in \cite{DRS14} on the sub-regime $(-\infty, \overline h)$ for $\overline h\leq h_\star$. 
\cite[Remark 2.9]{DRS14} conjectures that
$\overline h= h_\star \in (0,\infty)$ 
in all $d\geq 3$ 
and as before, we will also assume that 
$$
h\in (-\infty, \overline h),
$$
which guarantees that the level set $E^{\geq h}$ has a unique infinite connected component $\mathcal C_\infty$ and as usual, we will write $\Omega_0=\{0\in\mathcal C_\infty\}$ and will work with the conditional measure 
$$
\P_0=\P(\cdot| 0\in\mathcal C_\infty).
$$

\subsection{The simple random walk on the percolation models.}\label{sec-SRWPC}

We now define a (discrete time) simple random walk on the unique supercritical percolation cluster $\mathcal C_\infty$ corresponding to the percolation models 
discussed in the last section. 

Let a random walk start at the origin and at each unit of time,
the walk moves to a nearest neighbor site chosen uniformly at random from the accessible neighbors. 
More precisely, for each $\omega\in \Omega_0$, $x\in \Z^d$ and $e\in \mathcal U_d$, we set
\begin{equation}\label{pidef}
\pi_\omega(x,e)=\frac{ \1_{\{\omega_e=1\}} \circ \tau_x}{\sum_{{e^\prime \in \mathcal U_d}} \1_{\{\omega_{e^\prime}=1\}} \circ \tau_x} \in [0,1],
\end{equation}
and define a simple random walk $X=(X_n)_{n\geq 0}$ as a Markov chain taking values in $\Z^d$ with the transition probabilities
\begin{equation}\label{pitransit}
\begin{aligned}
&P^{\pi,\omega}_{0}(X_0=0)=1,\\
&P_{0}^{\pi,\omega}\big(X_{n+1}= x+e\big| X_n=x\big)=\pi_\omega(x,e).
\end{aligned}
\end{equation}
This is a canonical way to ``put" the Markov chain on the infinite cluster $\mathcal C_\infty$. Henceforth, we will refer to this Markov chain as the {\it{simple random walk on the percolation cluster}} (SRWPC). 

Let us remark that in the expression of $\pi_\omega(x,e)$ as well as $P_{0}^{\pi,\omega}$ we have apparently used the notation for bond percolation models appearing in Section \ref{sec-intro-bond}.
Very similar expression can be used for these objects pertaining to the site percolation models introduced in Section \ref{sec-intro-site} too. To alleviate notation, throughout the rest of the article we will continue to write the expressions \eqref{pidef} and \eqref{pitransit} for the transition kernels $\pi_\omega(x,e)$ and transition probabilities $P_{0}^{\pi,\omega}$ for the SRWPC corresponding to all the percolation models.


\section{Main results}\label{sec-results}
In Section \ref{sec-results-1} we will introduce the environment Markov chain, its empirical measures and certain relative entropy functionals which will be used later.
In Section \ref{sec-results-2}, we will announce our main results. In Section \ref{sec-results-3} we will carry out a sketch of the existing proof technique related to elliptic RWRE (\cite{Y08},\cite{R06}),
comment on the approach taken in the present paper regarding SRWPC and underline the differences to the earlier approach.

\subsection{The environment Markov chain.}\label{sec-results-1}
For each $\omega\in \Omega_0$, we consider the process $(\tau_{X_n}\omega)_{n\geq 0}$ which is a Markov chain taking values in the space of environments
$\Omega_0$. This is the {\it{environment seen from the particle}} and it plays an important r\^ole in the present context, see section 3.1 for a detailed description.
We denote by
\begin{equation}\label{localtime}
\mathfrak L_n= \frac 1n \sum_{k=0}^{n-1} \delta_{\tau_{X_k}\omega,{X_{k+1}-X_{k}}}
\end{equation}
the empirical measure of the environment Markov chain and the nearest neighbor steps of the SRWPC $(X_n)_{n\geq 0}$. This is a random element of $\Mcal_1(\Omega_0 \times  \mathcal U_d)$, the space of probability measures on $\Omega_0\times \mathcal U_d$. Note that $\Omega_0\times\mathcal U_d$ inherits the induced product topology from $\Omega\times \mathcal U_d= \{0,1\}^{\mathbb B_d}\times\mathcal U_d$,
while $\Mcal_1(\Omega_0 \times  \mathcal U_d)$ is equipped with the usual weak topology, with convergence being determined by convergence of integrals against continuous and bounded functions $f$ on $\Omega_0\times\mathcal U_d$. Note that under the weak topology $\Mcal_1(\Omega_0 \times  \mathcal U_d)$ is compact ($\Omega_0 \subset \Omega$ is closed and hence also compact). 
The empirical measures $\mathfrak L_n$ were introduced and their large deviation behavior (in the {\it{quenched setting}}) for elliptic random walks in random environments were studied by Yilmaz (\cite{Y08}). 

%
 
We note that, via the mapping $(\omega,e)\mapsto (\omega, \tau_e\omega)$ the space $\Mcal_1(\Omega_0 \times  \mathcal U_d)$ is embedded into $\Mcal_1(\Omega_0\times \Omega)$, 
and hence, every element $\mu\in \Mcal_1(\Omega_0\times \mathcal U_d)$ can be thought of as the {\it{pair empirical measure}} of the environment Markov chain. In this terminology,
we can define its marginal distributions by 
\begin{equation}\label{marginals}
\begin{aligned}
&\d (\mu)_1(\omega)= \sum_{e\in \mathcal U_d} \d \mu(\omega,e), \\
& \d (\mu)_{2}(\omega)= \sum_{\omega^\prime\colon \, \tau_e\omega^\prime=\omega} \d \mu(\omega^\prime,e)=\sum_{e\in \mathcal U_d} \d \mu(\tau_{-e}\omega,e).
\end{aligned}
\end{equation}
Here $(\mu)_1$ is a measure on $\Omega_0$ and $(\mu)_2$ is a measure on $\Omega$. 
A relevant subspace of $\Mcal_1(\Omega_0 \times  \mathcal U_d)$ is given by  
\begin{equation}\label{relevantmeasures}
\begin{aligned}
\Mcal_1^\star=\Mcal_{1}^ {\star}(\Omega_0\times \mathcal U_d)&=\bigg\{\mu\in\Mcal_1(\Omega_0 \times  \mathcal U_d)\colon\, (\mu)_1=(\mu)_ 2\ll \P_0 \, \,\mbox{and}\, \,\P_0\mbox{- almost surely,}\\
&\qquad\qquad\frac{\d \mu(\omega,e)}{\d (\mu)_{ 1}(\omega)} >0 \,\,
\mbox{if and only if}\,\,\omega_e=1\,\mbox{for}\, e\in \mathcal U_d\bigg\}.
\end{aligned}
\end{equation}
We remark that, here $(\mu)_1=(\mu)_ 2$ means that $(\mu)_ 2$ is supported on $\Omega_0$ and $(\mu)_1=(\mu)_ 2$. 
Furthermore, Lemma \ref{onetoone}  shows that elements in $\Mcal_1^\star$
are in one-to-one correspondence to Markov kernels (w.r.t. the environment process) on $\Omega_0$ which admit invariant probability measures which are absolutely continuous with respect to $\P_0$. 

Finally, we define a {\it{relative entropy functional}} $\mathfrak I: \Mcal_1(\Omega_0 \times  \mathcal U_d) \rightarrow [0,\infty]$ via
\begin{equation}\label{Idef}
\mathfrak I(\mu) =
\begin{cases}
\int_{\Omega_0} 
\sum_{e\in \mathcal U_d} \d\mu(\omega,e) \log \frac{\d \mu(\omega,e)}{\d(\mu)_{1}(\omega) \pi_\omega(0,e)} \quad\mbox{if}\,\mu\in\Mcal_1^\star,
\\
\infty\qquad\qquad\qquad\qquad\qquad\qquad\qquad\qquad\quad\mbox{else.}
\end{cases}
\end{equation}
For every continuous, bounded and real valued function $f$ on $\Omega_0 \times  \mathcal U_d$, we denote by
$$
\mathfrak I^\star(f)= \sup_{\mu\in\Mcal_1(\Omega_0 \times  \mathcal U_d)} \big\{ \langle f,\mu\rangle - \mathfrak I(\mu)\big\}
$$
the {\it{Fenchel-Legendre transform}} of $\mathfrak I(\cdot)$. Likewise, for every $\mu\in \Mcal_1(\Omega_0 \times  \mathcal U_d)$, $\mathfrak I^{\star\star}(\mu)$ denotes the {\it{Fenchel-Legendre transform}} of $\mathfrak I^\star(\cdot)$. 

\subsection{Main results: Quenched large deviation principle.}\label{sec-results-2} 

We are now ready to state the main result of this paper, which proves a large deviation principle for the distributions $P^{\pi,\omega}_0 \mathfrak L_n^{-1}$ on $\Mcal_1(\Omega_0\times\mathcal U_d)$ (usually called {\it{level-2}} large deviations) and the distributions $P^{\pi,\omega}_0 {\frac {X_n}{n}}^{-1}$ on $\R^d$ (usually called {\it{level-1}} large deviations). Both statements hold true for $\P_0$- almost every $\omega\in \Omega_0$ and in the case of elliptic RWRE, these already exist in the literature (see Yilmaz \cite{Y08} for level-2 large deviations and Rosenbluth \cite{R06} for level-1 large deviations) with the assumption which requires the $p$-th moment of the logarithm of the RWRE transition probabilities to be finite, for $p>d$.
In the present context, due to zero transition probabilities of the SRWPC, we necessarily have to drop this moment assumption.

Before we announce our main result precisely, let us remind the reader that all the percolation models that were required to satisfy Assumptions 1-5 in Section \ref{sec-intro-models} or were specifically introduced in Section \ref{sec-intro-bond} and Section \ref{sec-intro-site}, are assumed to be supercritical, the origin is always contained in the unique infinite cluster $\mathcal C_\infty$, $\P_0=\P(\cdot| \{0\in\mathcal C_\infty\})$ denotes the conditional environment measure and $P^{\pi,\omega}_0$
stands for the transition probabilities for SRWPC defined in \eqref{pitransit}. 
Here is the statement of our first main result.
\begin{theorem}[Quenched LDP for the pair empirical measures]\label{thmlevel2}
Let $d\geq 2$. 
Then for $\P_0$- almost every $\omega\in \Omega_0$, the distributions  
of $\mathfrak L_n$ under $P^{\pi,\omega}_0$
 satisfies a large deviation principle in the space of probability measures on
$\Mcal_1(\Omega_0 \times  \mathcal U_d)$ equipped with the weak topology.
The rate function $\mathfrak I^{\star\star}$ is the double Fenchel-Legendre transform of the functional $\mathfrak I$ defined in \eqref{Idef}. Furthermore, $\mathfrak I^{\star\star}$ is convex and has compact level sets.
\end{theorem}
In other words, for $\P_0$- almost every $\omega\in \Omega_0$,
\begin{equation}\label{ldpub}
\limsup_{n\to\infty} \frac 1n \log P^{\pi,\omega}_0 \big(\mathfrak L_n\in \mathcal C\big) \leq -\inf_{\mu\in \mathcal C} \mathfrak I^{\star\star}(\mu) \quad \forall\,\,\mathcal C\subset \Mcal_1(\Omega_0 \times  \mathcal U_d)\,\,\mbox{closed},
\end{equation}
and 
\begin{equation}\label{ldplb}
\limsup_{n\to\infty} \frac 1n \log P^{\pi,\omega}_0 \big(\mathfrak L_n\in \mathcal G\big) \geq -\inf_{\mu\in \mathcal G} \mathfrak I^{\star\star}(\mu) \quad \forall\,\,\mathcal G\subset \Mcal_1(\Omega_0 \times  \mathcal U_d)\,\,\mbox{open}.
\end{equation}
\noindent A standard computation shows that the functional $\mathfrak I$ defined in \eqref{Idef} is convex on $\Mcal_1(\Omega_0\times\mathcal U_d)$. The following lemma, whose proof is based on the ``zero speed regime" of the SRWPC under a supercritical drift and is deferred to until Section \ref{sec-6}, shows that 
$\mathfrak I^{\star\star}\ne \mathfrak I$.
\begin{lemma}\label{nonlsc}
Let $d\geq 2$. 
Then $\mathfrak I$ is not lower-semicontinuous on $\Mcal_1(\Omega_0\times \mathcal U_d)$. Hence, $\mathfrak I\ne \mathfrak I^{\star\star}$.
\end{lemma}

We remark that Theorem \ref{thmlevel2} is an easy corollary to the existence of the limit
$$
\lim_{n\to\infty} \frac 1n \log E^{\pi,\omega}_0 \big\{\exp\{n\big \langle f, \mathfrak L _n\big\rangle\big\}\big\}
=\lim_{n\to\infty} \frac 1n \log E^{\pi,\omega}_0 \bigg\{\exp\bigg( \sum_{k=0}^{n-1}f\big(\tau_{X_k}\omega, X_k-X_{k-1}\big)\bigg)\bigg\},
$$
for every continuous, bounded function $f$ on $\Omega_0 \times  \mathcal U_d$ and the symbol $\langle f, \mu\rangle$ denotes, in this context, the integral $\int_{\Omega_0} \d \P_0 (\omega) \sum_{e\in \mathcal U_d} f(\omega,e) \d \mu(\omega,e)$.
We formulate it as a theorem.
\begin{theorem}[Logarithmic moment generating functions]\label{thmmomgen}
For $d\geq 2$, $p> p_c(d)$ and every continuous and bounded function $f$ on $\Omega_0 \times  \mathcal U_d$,
$$
\lim_{n\to\infty} \frac 1n \log E^{\pi,\omega}_0 \bigg\{\exp\bigg( \sum_{k=0}^{n-1}f\big(\tau_{X_k}\omega, X_k-X_{k-1}\big)\bigg)\bigg\} = \sup_{\mu\in \Mcal_{1}^{\star}} \big\{\langle f,\mu\rangle- \mathfrak I(\mu)\big\} \quad\P_0-\mbox{a.s.}
$$
\end{theorem}
We will first prove Theorem \ref{thmmomgen} and deduce Theorem \ref{thmlevel2} directly.

Note that via the contraction map $\xi: \Mcal_1(\Omega_0 \times  \mathcal U_d) \longrightarrow \R^d$, 
$$
\mu\mapsto \int_{\Omega_0} \sum_e \, e\,\d\mu(\omega,e),
$$
we have $\xi(\mathfrak L_{n})= \frac{X_n-X_0}n= \frac {X_n} n$. Our second main result is the following corollary to Theorem \ref{thmlevel2}.

\begin{cor}[Quenched LDP for the mean velocity of SRWPC]\label{thmlevel1}
Let $d\geq 2$. 
Then the distributions $P^{\pi,\omega}_0\big(\frac {X_n}n\in \cdot\big)$ satisfy a large deviation principle 
with a rate function 
$$
\begin{aligned}
 J(x)&= \inf_{\mu\colon \xi(\mu)=x} \mathfrak I(\mu)\qquad x\in\R^d.
\end{aligned}
$$
\end{cor}

\begin{remark}\label{rmk-Kubota}
Note that Corollary \ref{thmlevel1} has been obtained by Kubota (\cite{K12}) for the SRWPC based on the method of Zerner (\cite{Z98}). 
Kubota used sub-addtivity and overcame the lack of the moment criterion of Zerner
by using classical results about the geometry of the percolation. 
This way he obtained a rate function which is convex and is given by the Legendre transform of the {\it{Lyapunov exponents}} derived by Zerner (\cite{Z98}). 
However, using the sub-additive ergodic theorem one does not get any expression or formula for the rate function, nor does the sub-additivity
seem amenable for deriving a {\it{level 2}} quenched LDP as in Theorem \ref{thmlevel2}.
\end{remark}

\begin{remark}\label{rmk-Mourrat}
 Mourrat (\cite{M12}) also considered level-1 quenched large deviation principle for a model of random walk in random potential containing the non-elliptic case, by taking a strategy similar to \cite{Z98}. 
Note that the framework in \cite{M12} gives equal probability with each path of a fixed length in an infinite cluster, 
so the random walk can be regarded as a Markov chain on the {\it augmented} space by adding a cemetery point to $\mathbb{Z}^d$ (see \cite{Z98-I}).
As we will see, our arguments will rely on the random walk being a Markov chain on an infinite percolation cluster $\mathcal{C}_{\infty}$,
and it will be intiguing to consider extensions of our results to the random walk in random potential (RWRP) framework considered in \cite{M12}.
Furthermore, given the broad range of models covered in the present paper, it will also be 
interesting to consider potentials that are only invariant and ergodic w.r.t. spatial shifts,
while dropping the i.i.d. requirement imposed in \cite{M12}. 
However, in order to derive large deviation principle for random walks on such random environments, 
it is desirable to have good chemical distance estimates on the infinite cluster.
\end{remark}

\subsection{Survey of earlier proof technique in the elliptic case and comparison with our method.}\label{sec-results-3}

Earlier relevant work for quenched large deviations was carried out by Kosygina-Rezakhanlou-Varadhan (\cite{KRV06}) for elliptic diffusions in a random drift.
Rosenbluth (\cite{R06}) first adapted this approach to the case of elliptic RWRE and derived a level-1 quenched large deviation principle for the distribution of the mean-velocity 
(the so-called {\it{level-1}} large deviations, recall Corollary \ref{thmlevel1}). Yilmaz (\cite{Y08}) then extended Rosenbluth's work on elliptic RWRE to a finer large deviation result for the 
pair empirical measures of the environment Markov chain (the so-called {\it{level-2}} large deviations, recall Theorem \ref{thmlevel2}). 
In the present case of deriving similar level-2 quenched large deviations for SRWPC, as a guiding philosophy, we also follow the main steps of Yilmaz (\cite{Y08}). 
However, due to fundamental obstacles that come up in several facets stemming from the inherent non-ellipticity of the percolation models,
an actual execution of the existing method \cite{Y08} fails for the present case of SRWPC.
In order to put our present work in context, in this section we will present a brief survey on the existing method that treated the elliptic case of RWRE (\cite{Y08}),
and to emphasize the similarities and differences of our approach to the earlier one, 
and we will also provide a comparative description of the main strategy for the proof of Theorem \ref{thmmomgen} 
that allows the treatment of models that are non-elliptic (like SRWPC), while simplifying 
the earlier proof technique used for the elliptic case. 
This will also underline the technical novelty of the present work.


\subsection{Comparison of our proof techniques with the earlier approach used for elliptic RWRE:}
As mentioned before, the purpose of the present subsection is to compare the proof techniques in the present paper to those
of previous work on elliptic RWRE (\cite{KRV06,R06,Y08}).  In particular, and unlike the rest of the paper, the
current subsection is intended for readers familiar with the techniques and ideas of those papers.
To keep notation consistent, in this survey we will continue to denote by $\P$ the law of a stationary and ergodic random environment and by $\pi(\omega,\cdot)$ we will denote the random walk transition probabilities in the random environment.
One of the requirements under which earlier results concerning elliptic RWRE (\cite{R06,Y08}) is the {\it{moment condition}} requiring $\int |\log\pi|^{d+\eps}\d\P<\infty$ for some $\eps>0$. 
The crucial argument is the existence of the limiting logarithmic moment generating function (recall Theorem \ref{thmmomgen}) whose proof splits into three main steps:

\noindent {\it{Lower bound.}} For models in elliptic RWRE, the lower bound part is based on a classical change of measure argument for the environment Markov chain, followed by an application of an ergodic theorem for the tilted Markov chain. This ergodic theorem is standard (see Kozlov \cite{K85}, Papanicolau-Varadhan \cite{PV81}) in the elliptic case where the (tilted) Markov chain transition probabilities are assumed to be strictly positive (as in the case studied in \cite{Y08}).

In the current case of SRWPC, the lower bound also follows the standard method of tilting the environment Markov chain as the elliptic RWRE case. However, for the tilted environment Markov chain for the percolation models,  the requisite ergodic theorem needs to be extended to the non-elliptic case which is the content of Theorem \ref{ergodicthm}.

\noindent{\it{Upper bound.}} For the elliptic RWRE case, the upper bound part of the proof  starts with a ``perturbation" of the exponential moment of the pair empirical measures $\mathfrak L_n$ defined in \eqref{localtime}.
This perturbation comes from integrating certain  ``gradient functions"  w.r.t. the local times $\mathfrak L_n$, and these gradient functions are intrinsically defined by the spatial action of the translation group $\Z^d$ on the environment space. In the elliptic case (\cite{Y08}, \cite{R06} and \cite{KRV06}), the class $\mathcal K$ of such gradient functions $F\in\mathcal K$ are required to satisfy the {\it{closed loop condition}} that underlines their gradient structure, a moment condition that requires $F\in L^{d+\eps}(\P)$,
and a mean-zero condition that demands $\E^{\P}[F]=0$.  Any such $F\in \mathcal K$ leads to its {\it{corrector}} $V_F(\omega,x)=\sum_{j=0}^{n-1} F(\tau_{x_j}\omega,x_{j+1}-x_j)$ which 
is defined as the  integral of the gradient $F$ along any path $x_0,\dots,x_n=x$ between two fixed points $x_0$ and $x_n$. Note that the choice of the path does not influence the integral, thanks to the closed loop condition imposed on $F$. For any $F\in\mathcal K$, Rosenbluth (\cite{R06}) then proved that, the corresponding corrector $V_F$ has a ``sub-linear growth at infinity". Roughly speaking, this means, $\P$-almost surely, $|V_F(\omega,x)|=o(|x|)$ as $|x|\to\infty$. This is a crucial technical step in Rosenbluth's work that is proved adapting the original approach of \cite{KRV06} involving Sobolev embedding theorem and invoking Garsia-Rodemich-Rumsey estimate, and his the proof there hinges on the moment condition $F\in L^{d+\eps}(\P)$.$^1$ \footnote{$^{1}$Recall that the elliptic random environment is also required to satisfy the  moment condition $\E^{\P}[|\log\pi|^{d+\eps}]<\infty$.}
Since for elliptic RWRE, $\P$ is invariant w.r.t. the translations, one then exploits the mean-zero condition of the gradients $F$ and invokes the ergodic theorem to get the desired sub-linearity property. This property implies, in particular, that the effect of the aforementioned perturbation by the corrector $V_F$ in the exponential moment is indeed negligible. This is the crucial argument for the upper bound part for the existing literature on elliptic RWRE. 

Now for the upper bound part for SRWPC, already the aforementioned moment condition of the elliptic case fails (zeroes of SRWPC transition probabilities $\pi$ already make the first moment $\E_0(|\log\pi|)$ possibly infinite). Hence, we are not entitled to follow the method of Rosenbluth (\cite{R06}, see also \cite{GRSY13}) for proving the sub-linear growth property of the correctors.
Moreover, the crucial mean-zero condition required in the elliptic case also fails for percolation due to the fundamental fact that the {\it{spatial action of the shifts $\tau_e$ on $\Omega_0$ is not $\P_0$-measure preserving}}. The lack of these two properties requires that we reformulate the conditions on our class of gradients. Besides the closed loop property in the infinite cluster, we demand uniform boundedness of the gradients in $\P_0$-norm and the validity of an ``induced mean-zero property" to circumvent the above mentioned non-invariant nature of the spatial shifts $\tau_e$ w.r.t. $\P_0$, see Section \ref{subsec-classG} for details. With these assumptions, we prove the requisite "sub-linear growth" property of the correctors corresponding to our gradients, see Theorem \ref{sublinearthm}. Our approach for proving this sub-linearity property is  therefore different from  the existing literature (\cite{R06}, \cite{GRSY13}). Instead, it is based on techniques from ergodic theory, combined with geometric arguments that capture precise control on  the ``chemical distance" (or the geodesic distance) between two points $x$ and $y$ in the infinite cluster $\mathcal C_\infty$ (proved in Lemma \ref{chemdist}), as well as exponential tail bounds for the shortest distance between the origin and the first arrival of the cluster in the positive parts of the co-ordinate axes (proved in Lemma \ref{lemma-ell}). Given the above sub-linear growth property on the infinite cluster which holds the pivotal argument,  we then carry out the same ``corrector perturbation" approach as in the elliptic case to the desired upper bound property, see Lemma \ref{ub}.

\noindent{\it{Equivalence of lower and upper bounds.}} Having established both lower and upper bounds, one then faces the task of matching these two bounds. In the case of elliptic diffusions with a random drift, a
seminal idea was introduced in \cite{KRV06} by applying convex variational analysis followed by applications of certain min-max theorems. The success of this ``min-max" approach relies on, among other requirements, ``compactness" of the underlying variational problem. In the elliptic case, this can be achieved by truncating the variational problem at a finite level which allows the application of the min-max theorems, followed by an approximation procedure by letting the truncation level to infinity. In the lattice, i.e., for elliptic RWRE a similar idea was used (\cite{Y08}, \cite{R06}) in order to use the min-max argument. Indeed, by restricting the variational problem to a finite region in the environment space $\Omega$ and taking conditional expectation w.r.t. a finite $\sigma$-algebra $B_k$, \cite{Y08} then used the min-max theorems for every fixed $k$. Roughly speaking, this leads to the study of conditional expectations 
\begin{equation}\label{def-Fk}
F_k:=\E\big[f_k-f_k\circ \tau_e| B_{k-1}\big],
\end{equation}
for test functions $f_k$, and one needs to prove that $F_k\to F$ as $k\to\infty$ such that $F\in \mathcal K$ (where $\mathcal K$ is the class of gradients with the required properties discussed in the upper bound part). Note that, for every fixed $k$, $F_k$ is not a gradient. However, exploiting the underlying assumption $\E^\P[|\log\pi|^{d+\eps}]<\infty$, one shows that $\{F_k\}_k$ remains uniformly bounded in $L^{d+\eps}(\P)$ so that one can take a weak limit $F$. After successive application of the tower property for the conditional expectations, one then proves that the limit $F$ is indeed a gradient (i.e., satisfies the aforementioned closed loop condition), $F\in L^{d+\eps}(\P)$. Furthermore, $\E_\P[F]=0$, which readily comes for free from \eqref{def-Fk} and the {\it{invariant action}} of $\tau_e$ w.r.t. the environment law $\P$. In particular, $F\in\mathcal K$ and modulo some technical work, this fact also matches the lower and upper bound of the limiting logarithmic moment generating function for the elliptic RWRE case.

\noindent Now for the ``equivalence of bounds" for SRWPC, one can also try to emulate the strategy of (\cite{Y08}, \cite{R06}) by carrying out the same convex variational analysis and 
applying the same min-max theorems by restricting to a finite region and conditional on a finite $\sigma$-algebra $B_k$. However,  taking the conditional expectation as in \eqref{def-Fk} 
w.r.t. $\E_0$ any attempt towards deriving the requisite properties stated in Section \ref{subsec-classG} of the limiting function $F$ completely fails. 
Note that in conditional expectation w.r.t. $\E_0$, one involves the measure $\P_0$ that is not preserved under the action of the shifts $\tau_e$. In particular, we are not entitled to use any tower property. Plus, conditioning w.r.t. a finite $\sigma$-algebra $B_k$ is incompatible for handling possibly long excursions of the infinite cluster before hitting the coordinate axes on each direction, which is a crucial issue one has to handle in order to prove the requisite induced mean-zero property of our limiting gradient.

\noindent Therefore, for the equivalence of bounds, we take a different route based on an {\it{entropy coercivity}} and {\it{entropy penalization}} method, which constitutes Section \ref{sec-proof-ldp}.
This approach seems to be more natural in that it exploits the built-in structure of relative entropies that is already present in the underlying variational formulas. We make use of the coercivity property of the relative entropies in Lemma \ref{thm-lbub-lemma1} and Lemma \ref{thm-lbub-lemma2} to overcome the lack of the compactness in our variational analysis. One advantage of this method is that our variational analysis leads to the study of {\it{gradients}} directly, where we can work with functions
\begin{equation}\label{def-Gn-intro}
G_n(\omega,e)= g_n(\omega)-g_n(\tau_e\omega),
\end{equation}
on the infinite cluster (see Lemma \ref{lemma-last}), instead of relying on conditional expectations like in \eqref{def-Fk}.
Given the gradient structure of $G_n$, and the estimates proved in Lemma \ref{chemdist} and \ref{lemma-ell}, our analysis then also shows that the limiting gradients satisfy all  the desired properties formulated in Section \ref{subsec-classG} (see Lemma \ref{lemma-last}) and the lower and upper bounds are readily matched. 
We also remark that the argument in our approach works equally well for the elliptic RWRE model considered before, see Remark \ref{rmk-simplify}. In particular, our method completely avoids the 
tedious effort needed in the earlier approach through the use of conditional expectations, tower property and Mazur's theorem in order to show that the limit of $F_k$ defined in \eqref{def-Fk} is a gradient, and the equivalence of upper and lower bounds. In our approach, any weak limit of $G_n$ defined in \eqref{def-Gn-intro} is immediately a gradient and this readily makes the lower and the upper bound match (again, it is imperative here that we can work with $G_n$ which is itself a gradient, unlike \eqref{def-Fk}). We refer to \cite[Sect.3.3]{R06} or \cite[Sect.2.1.3]{Y08} for a comparison with our approach in proving Theorem \ref{thm-lbub}.

\begin{remark}[Differences to the Kipnis-Varadhan corrector]\label{KVremark}
Let us finally remark that the class of gradient functions introduced in Section \ref{subsec-classG} share some similarities to the gradient of Kipnis-Varadhan corrector which is a central object of interest for reversible random motions in random media. Particularly for SRWPC this is crucial for proving a quenched central limit theorem (\cite{SS04}, \cite{MP07}, \cite{BB07}, \cite{PRS15})-- the corrector expresses the deformation caused by a harmonic embedding of the random walk in the infinite cluster in $\R^d$, and modulo this deformation, the random walk becomes a martingale. 
 However, our gradient functions that are defined in Section \ref{subsec-classG} are structurally different from the gradient of the Kipnis-Varadhan corrector. 
 Though they share similar properties as {\it{gradients}}, our gradients  miss the above mentioned  {\it{harmonicity}} property enjoyed by the Kipnis-Varadhan corrector. 
 This can be explained by the fact that large deviation lower bounds are based on a certain {\it{tilt}} which spoils any inherent reversibility of the model, which is a crucial base of Kipnis-Varadhan theory.\qed
 \end{remark}

The rest of the article is organized as follows. 
In Section \ref{sec-lb}, Section \ref{sec-classG} and Section \ref{sec-proof-ldp} we prove the lower bound, the upper bound and the equivalence of bounds for Theorem \ref{thmmomgen}, respectively. Section \ref{sec-6} is devoted to the proofs of Theorem \ref{thmmomgen}, Theorem \ref{thmlevel2}, Corollary \ref{thmlevel1} and Lemma \ref{nonlsc}.

\section{Lower bounds of Theorem \ref{thmlevel2} and Theorem \ref{thmmomgen}}\label{sec-lb}
We first introduce a class of environment Markov chains for SRWPC and prove an ergodic theorem for these in Section \ref{sec-ergthm}. We then derive the lower bounds
for Theorem \ref{thmlevel2} and Theorem \ref{thmmomgen} in Section \ref{subsec-lb}.
\subsection{An ergodic theorem for Markov chains on non-elliptic environments}\label{sec-ergthm}

In this section 
we need some input from the environment seen from the particle, which, with respect to a
suitably changed measure, possesses important ergodic properties. 

 
Recall that, given the transition probabilities $\pi$ from \eqref{pidef}, for $\P_0$- almost every $\omega\in \Omega_0$, the process $(\tau_{X_n}\omega)_{n\geq 0}$ is a Markov chain with transition kernel
$$
(R_\pi g)(\omega)= \sum_{e\in\mathcal U_d} \pi_\omega(0,e) g(\tau_e \omega),
$$
for every function $g$ on $\Omega_0$ which is measurable and bounded.

We need to introduce a class of transition kernels on the space of environments. We denote by $\widetilde \Pi$ the space of functions $\tilde\pi: \Omega_0\times \mathcal U_d\rightarrow [0,1]$  
which are measurable in $\Omega_0$, $\sum_{e\in \mathcal U_d} \tilde\pi(\omega,e)=1$ for almost every $\omega\in \Omega_0$ and for every $\omega\in \Omega_0$ and $e\in \mathcal U_d$, 
\begin{equation}\label{pitilde}
\tilde\pi(\omega,e)=0 \,\,\,\mbox{if and only if}\,\,\, \pi_\omega(0,e)=0.
\end{equation}
For every $\tilde\pi\in\widetilde\Pi$ and $\omega\in \Omega_0$, we define the corresponding quenched probability distribution of the Markov chain $(X_n)_{n\geq 0}$ by
\begin{equation}
\begin{aligned}
&P^{\tilde\pi,\omega}_0(X_0=0)=1\\
&P^{\tilde\pi,\omega}_0(X_{n+1}=x+e| X_n=x)= \tilde\pi(\tau_x\omega,e). 
\end{aligned}
\end{equation}

With respect to every $\tilde\pi\in \widetilde \Pi$ we also have a transitional kernel 
$$
(R_{\tilde \pi} g)(\omega)= \sum_{e\in\mathcal U_d} \tilde\pi(\omega,e) g(\tau_e \omega),
$$
for every measurable and bounded $g$. 
For every measurable function $\phi\geq 0$ with $\int \phi \d \P_0=1$, we say that the measure $\phi\d \P_0$ is $R_{\tilde \pi}$-invariant, or simply $\tilde\pi$-invariant, if,
\begin{equation}\label{invdensity}
\phi(\omega)= \sum_{e\in\mathcal U_d} \tilde\pi\big(\tau_{-e}\omega, e\big) \phi\big(\tau_{-e}\omega\big).
\end{equation}
Note that in this case,
\begin{equation}\label{invdensity_g}
\int g(\omega)\phi(\omega)d\P_0(\omega) = \int (R_{\tilde \pi} g) (\omega)\phi(\omega)d\P_0(\omega),
\end{equation}
for every bounded and measurable $g$.

We denote by $\mathcal E$ such pairs of $(\tilde \pi,\phi)$, i.e.,
\begin{equation}\label{ergodicpair}
\mathcal E= \bigg\{ (\tilde \pi, \phi)\colon \, \tilde\pi\in \widetilde\Pi, \phi\geq 0, \E_0(\phi)=1, \, \phi\d \P_0 \,\mathrm{is}\, \,\tilde\pi-\,\mathrm{invariant}\bigg\}.
\end{equation}
We need an elementary lemma which we will be using frequently. Recall the set $\Mcal_1^\star$ from \eqref{relevantmeasures}.
\begin{lemma}\label{onetoone}
There is a one-to-one correspondence between the sets $\Mcal_1^\star$ and $\mathcal E$.
\end{lemma}
\begin{proof}
Given arbitrarily $(\tilde\pi, \phi)\in \mathcal E$, we take
\begin{equation}\label{themap}
\begin{aligned}
\d \mu(\omega,e)= \tilde\pi(\omega,e) \phi(\omega) \,\d \P_0
 = \tilde\pi(\omega,e)  \bigg(\sum_{ \tau_e\omega^\prime=\omega} \tilde\pi(\omega^\prime,e) \phi(\omega^\prime)\bigg) \, \d \P_0.
\end{aligned}
\end{equation}
By \eqref{marginals}, $\P_0$-almost surely, 
$$
\begin{aligned}
\d(\mu)_1(\omega) =\sum_{e\in \mathcal U_d} \d \mu(\omega,e)
= \sum_{ \tau_e\omega^\prime=\omega} \tilde\pi(\omega^\prime,e) \phi(\omega^\prime) \, \d \P_0
= \sum_{ \tau_e\omega^\prime=\omega}\d \mu(\omega^\prime,e)
=\d (\mu)_2(\omega). 
\end{aligned}
$$
Hence, $(\mu)_1=(\mu)_2\ll \P_0$. Furthermore, if the edge $0\leftrightarrow e$ is present in the configuration $\omega$ (i.e., $\omega(e)=1$), then $\pi_\omega(0,e)>0$, and by 
our requirement  \eqref{pitilde},
$$
\frac{\d\mu(\omega,e)}{\d(\mu_1)(\omega)} = \tilde \pi(\omega,e)>0,
$$
Hence $\mu \in \mathcal M_1^\star$. Conversely, given arbitrarily  $\mu \in \mathcal M_1^\star$, we can choose
$(\tilde \pi, \phi)= (\frac {\d \mu}{\d (\mu)_1}, \frac {\d (\mu)_1}{\d \P_0})$ and readily check that $(\tilde \pi, \phi) \in \mathcal E$.
\end{proof}
We now state and prove the following ergodic theorem for the environment Markov chain under every transition kernel $\tilde\pi\in\widetilde\Pi$.
Theorem \ref{ergodicthm} is an extension of a similar statement (see Kozlov \cite{K85}, Papanicolau-Varadhan \cite{PV81}) that holds for elliptic transition kernels $\widetilde\pi(\cdot,e)$ to the non-elliptic case.
\begin{theorem}\label{ergodicthm}
Fix $\tilde\pi\in\widetilde\Pi$. If there exists a probability measure $\mathbb Q \ll \P_0$ which is $\tilde\pi$-invariant, then the following three implications hold:
\begin{itemize}
\item $\mathbb Q\sim \P_0$.
\item $\mathbb Q$ is ergodic for the environment Markov chain with transition kernel $\tilde\pi$.
\item There can be at most one such measure $\mathbb Q$.
\end{itemize}
In particular, every $\widetilde\pi$-invariant set of environments will have $\P_0$-measure $0$ or $1$, as $\mathbb Q\sim\P_0$, and $\mathbb Q$ is ergodic. 
\end{theorem}
\begin{proof}
We fix $\tilde\pi\in \widetilde\Pi$ and let $\mathbb Q\ll \P_0$ be $\tilde \pi$- invariant. We prove the theorem in three steps.

\noindent {\bf{Step 1:}} We will first show that, $\frac {\d \mathbb Q} {\d \P_0}>0$ $\P_0$- almost surely. This will imply that $\mathbb Q\sim \P_0$. 

Indeed, to the contrary, let us assume that, $0< \P_0(A) <1$ where $A= \big\{\omega\colon \frac {\d \mathbb Q} {\d \P_0}(\omega)>0\big\}$. Then, $\mathbb Q\sim \P_0(\cdot| A)$. 
If we sample $\omega_1 \in \Omega_0$ according to $\mathbb Q$ and $\omega_2$ according to $\tilde\pi(\omega_1,\cdot)$, 
then 
the distribution of $\omega_2$ is absolutely continuous with respect to $\mathbb Q$ (recall $\mathbb Q$ is $\tilde\pi$ invariant)  and thus, on $A^c$, the distribution of $\omega_2$ has zero measure.

This implies that, for almost every $\omega_1\in A$ and every $e\in \mathcal U_d$ such that $\tilde \pi(\omega_1,e)>0$, $\tau_e \omega_1 \in A$. Since $\tilde \pi \in\widetilde \Pi$, 
for almost every $\omega_1\in A$ and every $e\in \mathcal U_d$ such that $\pi(\omega_1,e)>0$, $\tau_e \omega_1 \in A$.
Now if we sample $\omega_1$ according to $\P_0(\cdot |A)$ and $\omega_2$ according to $\pi(\omega_1,\cdot)$, then, with
probability $1$, $\omega_2\in A$. In other words, $A$ is invariant under $\pi$ (more precisely, $A$ is invariant under the Markov kernel $R_\pi$). Since $\P_0$ is $\pi$-ergodic (see \cite[Proposition 3.5]{BB07}), $\P_0(A)\in \{0,1\}$. By our assumption, $\P_0(A)=1$.

\noindent {\bf{Step 2:}} Now we prove that the environment Markov chain with initial law $\mathbb Q$ and transition kernel $\tilde\pi$ is $\P_0$ ergodic. Let us assume on the contrary, that for some measurable $D$, $\mathbb Q(D)>0$, $\mathbb Q(D^c)>0$ and $D$ is $\tilde \pi$ invariant.
Hence $\P_0(D)>0$ and $\P_0(D^c)>0,$ 
by $\mathbb Q \sim \P_0$.
Further, the conditional measure $\mathbb Q_D(\cdot)= \mathbb Q(\cdot| D)$ is $\tilde\pi$ invariant and $\mathbb Q_D\ll \P_0$.
But $\mathbb Q_D(D^c)=0$ and hence, $\frac{\d\mathbb Q_D}{\d \P_0}(D^c)=0$. This contradicts the first step.  

\noindent {\bf{Step 3:}} We finally prove uniqueness of every $\mathbb Q$ which is $\tilde\pi$- invariant and absolutely continuous with
respect to $\P_0$. Let $\Omega^{\Z}$ be the space of the trajectories $(\dots,\omega_{-1},\omega_0,\omega_1,\dots)$ of the environment 
chain, $\mu_{\mathbb Q}$ the measure associated to the transition kernel $\tilde \pi$ whose finite dimensional distributions are given by
$$
\mu_\mathbb Q\big((\omega_{-n},\dots,\omega_n) \in A\big)= \int_A \mathbb Q(\d \omega_{-n}) \prod_{j=-n}^{n-1}\tilde\pi\big(\omega_j, \d \omega_{j+1}\big).
$$
for every finite dimensional cylinder set $A$ in $\Omega^\Z$. Let $T: \Omega^\Z \longrightarrow \Omega^\Z$ be the shift given by $(T\omega)_n= \omega_{n+1}$ for all $n\in \Z$. 
Since $\mathbb Q$ is $\tilde\pi$- invariant and ergodic, by Birkhoff's theorem,
$$
\lim_{n\to\infty} \frac{1}{n} \sum_{k=0}^{n-1} g \circ T^k = \int g \d \mu_\mathbb Q, 
$$
$\mu_\mathbb Q$ (and hence $\mu_{\P_0}$) almost surely for every bounded and measurable $g$ on $\Omega^\Z$. Since the environment chain $(\tau_{X_k}\omega)_{k\geq 0}$ has the same law w.r.t. $\int P^{\tilde\pi,\omega}_0 \d \mathbb Q$ as $(\omega_0,\omega_1,\dots)$ has w.r.t. $\mu_\mathbb Q$, if $f(\omega_0)= g(\omega_0,\omega_1,\dots)$, then 
$$
\lim_{n\to\infty} \frac 1 n\sum_{k=0}^{n-1} f \circ \tau_{X_k}= \lim_{n\to\infty} \frac 1 n\sum_{k=0}^{n-1} g \circ T^k = \int g \d \mu_\mathbb Q= \int f \d \mathbb Q,
$$
for every bounded and measurable $f$ on $\Omega$. 
The uniqueness of $\mathbb Q$ follows.
\end{proof}

\begin{cor}\label{ergcor}
For every pair $(\tilde\pi,\phi)\in \mathcal E$ and every continuous and bounded function $f: \Omega_0 \times  \mathcal U_d \rightarrow \R$, 
$$
\lim_{n\to\infty} \frac 1n \sum_{k=0}^{n-1} f(\tau_{X_k}\omega, X_{k+1}-X_k)= \int_{\Omega_0} \, \d\P_0\, \phi(\omega) \sum_e f(\omega,e) \tilde\pi(\omega,e),  \ \ \mbox{$\P_0 \times P^{\tilde \pi, \omega}_{0}$-a.s. }
$$
\end{cor}
\begin{proof}
This is an immediate consequence of Theorem \ref{ergodicthm} and Birkhoff's ergodic theorem.
\end{proof}

\subsection{Proof of lower bounds.}\label{subsec-lb}
We now prove the required lower bound \eqref{ldplb}. Its proof follows a standard change of measure argument and given Theorem \ref{ergodicthm}, although the argument is very similar to Yilmaz (\cite{Y08}),
we present this short proof for convenience of the reader and to keep the article self-contained. Recall the definition of $\mathfrak I$ from \eqref{Idef}.
\begin{lemma}[The lower bound]\label{lemmalb}
For every open set $\mathcal G$ in $\Mcal_1(\Omega_0 \times  \mathcal U_d)$, $\P_0$- almost surely,
\begin{equation}\label{eqlemmalb}
\begin{aligned}
\liminf_{n\to\infty} \frac 1n \log P^{\pi,\omega}_0 \big(\mathfrak L_n \in \mathcal G\big) &\geq - \inf_{\mu\in \mathcal G} \mathfrak I(\mu) \\
&=- \inf_{\mu\in \mathcal G} \mathfrak I^{\star\star}(\mu).
\end{aligned}
\end{equation}
\end{lemma}
\begin{proof}
For the lower bound in \eqref{eqlemmalb}, it is enough to show that, for every $\mu\in \Mcal_1^\star$ and every open neighborhood $\mathcal U$ containing $\mu$,
\begin{equation}\label{lb0}
\liminf_{n\to\infty} \frac 1n \log P^{\pi,\omega}_0\big(\mathfrak L_n \in \mathcal U\big) \geq - \mathfrak I(\mu).
\end{equation}
Given $\mu\in\Mcal_1^\star$, from Lemma \ref{onetoone} we can get the pair 
\begin{equation}\label{map}
(\tilde \pi, \phi)= \bigg(\frac {\d \mu}{\d (\mu)_1}, \frac {\d (\mu)_1}{\d \P_0}\bigg) \in \mathcal E,
\end{equation}
and by Theorem \ref{ergodicthm}, 
\begin{equation}\label{lb1}
\lim_{n\to\infty} P^{\tilde\pi,\omega}_0\big(\mathfrak L_n \in \mathcal U\big) =1.
\end{equation}
Further,
$$
\begin{aligned}
P^{\pi,\omega}_0\big(\mathfrak L_n \in \mathcal U\big) &= E^{\tilde\pi,\omega}_0 \bigg\{\1_{\{\mathfrak L_n \in \mathcal U\}} \frac{\d P_0^{\pi,\omega}}{\d P_0^{\tilde\pi,\omega}}\bigg\} \\
&=\int \d P^{\tilde\pi,\omega}_0 \bigg\{\1_{\{\mathfrak L_n \in \mathcal U\}} \exp\bigg\{-\log\,\frac{\d P_0^{\tilde\pi,\omega}}{\d P_0^{\pi,\omega}}\bigg\}\bigg\} .
\end{aligned}
$$
Hence, by Jensen's inequality, 
$$
\begin{aligned}
\liminf_{n\to\infty}\frac 1n \log P^{\pi,\omega}_0\big(\mathfrak L_n \in \mathcal U\big) 
&\geq \liminf_{n\to\infty}\frac 1n \log P^{\tilde\pi,\omega}_0\big(\mathfrak L_n \in \mathcal U\big) \\
&\qquad-\limsup_{n\to\infty}\frac 1{nP^{\tilde\pi,\omega}_0\big(\mathfrak L_n \in \mathcal U\big)} \int_{\{\mathfrak L_n \in \mathcal U\}} \d P^{\tilde\pi,\omega}_0 \bigg\{\log\,\frac{\d P_0^{\tilde\pi,\omega}}{\d P_0^{\pi,\omega}}\bigg\}\\
&= -\int \d\P_0(\omega)\,\phi(\omega) \sum_{|e|=1}  \tilde\pi(\omega,e)  \log \frac{\tilde\pi(\omega,e)}{\pi_\omega(0,e)}\\
&=- \mathfrak I(\mu),
\end{aligned}
$$
where the first equality follows from \eqref{lb1} and corollary \ref{ergcor} and the second equality follows from \eqref{map}. This proves \eqref{lb0}. Finally, since 
$\mathcal G$ is open, $\inf_{\mu\in \mathcal G} \mathfrak I(\mu)= \inf_{\mu\in \mathcal G} \mathfrak I^{\star\star}(\mu)$ (see \cite{R70}). This proves the equality in \eqref{eqlemmalb} and the lemma.
\end{proof}
We now prove the lower bound for the limiting logarithmic moment generating function required for Theorem \ref{thmmomgen}.
\begin{cor}\label{corlb}
For every continuous and bounded function $f: \Omega_0 \times  \mathcal U_d\longrightarrow \R$ and for $\P_0$-almost every $\omega\in \Omega_0$, 
\begin{equation}\label{lb}
\begin{aligned}
\liminf_{n\to\infty} \frac 1n\log E^{\pi,\omega}_0 \bigg\{ \exp \bigg(\sum_{k=0}^{n-1} f\big(\tau_{X_k}\omega, X_{k+1}-X_k\big)\bigg)\bigg\} &\geq \sup_{\mu\in \Mcal_{1}^{\star}} \big\{\langle f,\mu\rangle- \mathfrak I(\mu)\big\}\\
&=\sup_{\mu\in \Mcal_{1}(\Omega_0 \times  \mathcal U_d)} \big\{\langle f,\mu\rangle- \mathfrak I(\mu)\big\}.
\end{aligned}
\end{equation}
\end{cor}
\begin{proof}
This follows immediately from Varadhan's lemma and Lemma \ref{lemmalb}.
\end{proof}
We denote the variational formula in Corollary \ref{corlb} by
\begin{equation}\label{Lf}
\begin{aligned}
\overline H(f)&=\sup_{\mu\in \Mcal_1^\star} \big\{\langle f, \mu\rangle- \mathfrak I(\mu)\big\}\\
&= \sup_{(\tilde\pi,\phi)\in \mathcal E} \bigg\{\int \d\P_0(\omega)\,\phi(\omega) \sum_{|e|=1}  \tilde\pi(\omega,e) \bigg\{f(\omega,e) - \log \frac{\tilde\pi(\omega,e)}{\pi_\omega(0,e)}\bigg\}\bigg\},
\end{aligned}
\end{equation}
and recall from Lemma \ref{onetoone} the one-to-one correspondence between elements of the set $\Mcal_1^\star$ and the pairs $\mathcal E$ (see \eqref{map}, \eqref{Idef}).
For the variational analysis that follows in Section \ref{sec-proof-ldp}, it is convenient to write down a more tractable representation 
of the above variational formula. This is based on the following observation, which was already made by Kosygina-Rezakhanlou-Varadhan (\cite{KRV06}) and used by Yilmaz (\cite{Y08}) and Rosenbluth (\cite{R06}). 
Recall that by \eqref{invdensity_g}, if $(\phi,\tilde\pi)\in \mathcal E$, then for every bounded and measurable function $g$ on $\Omega_0$,
\begin{equation}\label{thm-lbub1}
\sum_e \int \phi(\omega) \tilde\pi(\omega,e) \big(g(\omega)- g(\tau_e\omega)\big)\d\P_0(\omega) = 0.
\end{equation}
On the other hand, if $(\phi,\tilde\pi)\notin \mathcal E$, then for some bounded and measurable function $g$ on $\Omega_0$, the above integral on the left hand side is non-zero. By taking constant multiples of such a function $g$, we see that
$$
\inf_{g}\, \int \phi(\omega) \sum_{e\in\mathcal U_d}\,\,\tilde\pi(\omega,e) \big(g(\omega)-g(\tau_e\omega)\big)\,\d\P_0(\omega)=
\begin{cases}
0 \qquad{\mbox{if}} \,\,(\phi,\tilde\pi)\in \mathcal E\\
-\infty \qquad{\mbox{else.}}
\end{cases}
$$
with the infimum being taken over every bounded and measurable function $g$. Hence, we can rewrite \eqref{Lf} as 
\begin{equation}\label{thm-lbub2a}
\begin{aligned}
\overline H(f) &=\sup_{\phi}\, \sup_{\tilde\pi\in\widetilde\Pi}\,\inf_{g} \bigg[\int \d\P_0(\omega)\,\phi(\omega) \bigg\{
\sum_{e\in \mathcal U_d}  \tilde\pi(\omega,e) 
 \bigg(f(\omega,e) -\log\frac{\widetilde\pi_\omega(0,e)}{\pi_\omega(0,e)}
+\big(g(\omega)-g(\tau_e\omega)\big)\bigg)\bigg\}\bigg] 
\end{aligned}
\end{equation}

\section{Upper bound for the proof of Theorem \ref{thmmomgen}}\label{sec-classG}
We will now introduce the class of relevant gradient functions in Section \ref{subsec-classG}, derive an important property 
of these  functions in Section \ref{subsec-sublinear} and prove the desired upper bound of Theorem \ref{thmmomgen} in Section \ref{subsec-ub}.

\subsection{The class $\mathcal G_\infty$ of gradients and the corresponding correctors}\label{subsec-classG}

We introduce a class of functions which will play an important role for the large deviation analysis to follow.
However, before introducing this class we need the notion of the {\em induced shift} on $\Omega_0$.


{Recall \eqref{def-n}.}
Then the {\em induced shift} is defined as 
\begin{equation}\label{def-sigmae}
\sigma_e(\omega)=\tau_{k(\omega,e)e}\,\,\omega.
\end{equation}
It is well-known that, for every $e\in\mathcal U_d$, $\sigma_e\colon \Omega_0\rightarrow \Omega_0$ is $\P_0$-measure preserving and ergodic (\cite[Theorem 3.2]{BB07}).
Furthermore, for every $k\in \N$, we inductively set
\begin{equation}\label{def-nk}
n_1(\omega,n)=k(\omega,e) \qquad n_{k+1}(\omega,e)=n_k(\sigma_e\omega,e).
\end{equation}

Now we turn to the definition of $\mathcal G_\infty$. We say that a function $G: \Omega_0 \times  \mathcal U_d \longrightarrow \R$ is in class $\mathcal G_\infty$ 
if it satisfies the conditions \eqref{unifbound}, \eqref{closedloop} and \eqref{meanzero} listed below:
\begin{itemize}
\item {\bf{Uniform boundedness.}} For every $e\in \mathcal U_d$, 
\begin{equation}\label{unifbound}
\mathrm{ess} \sup_{\P_0} G(\cdot,e) =A< \infty. 
\end{equation}
\item {\bf{Closed loop on the cluster.}} Let $(x_0,\dots,x_n)$ be a closed loop on the infinite cluster $\mathcal C_\infty$ (i.e., $x_0,x_1,\dots, x_n$ is 
 a nearest neighbor occupied path so that $x_0=x_n$). Then,
 \begin{equation}\label{closedloop}
 \sum_{j=0}^{n-1} G(\tau _{x_j} \omega, x_{j+1}-x_j) =0 \quad \P_0- \mbox{almost surely}.
 \end{equation}
 For every $G\in \mathcal G_\infty$, the closed loop condition has two important consequences. First, along every nearest neighbor occupied path $(x_0, x_1, \dots, x_m)$ so that $x_0=0$ and $x_m=x$ on $\mathcal C_\infty$, 
 for every $G(\cdot,\cdot)$ that satisfies \eqref{closedloop}, we can define the {\it{corrector}} corresponding to $G$ as
  \begin{equation}\label{Psidef}
 V(\omega,x)=  V_G(\omega,x)=  \sum_{j=0}^{m-1} G(\tau _{x_j} \omega, x_{j+1}-x_j). 
  \end{equation}
By \eqref{closedloop}, this definition is clearly independent of the chosen path for almost every $\omega\in \{x\in \mathcal C_\infty\}$. 
Also note that, for every $G$ that satisfies \eqref{closedloop}, $V=V_G$ satisfies the following 
 {\it{Shift covariance}} condition: For $\P_0$-almost every $\omega\in \Omega_0$ and all $x, y\in \mathcal C_\infty$,
\begin{equation}\label{def-shiftcov}
V(\omega,x)- V(\omega,y)= V(\tau_y\omega, x-y).
\end{equation}

\item {\bf{Zero induced mean:}} Recall the definition of $k(\omega,e)$ from \eqref{def-n} and write 
\begin{equation}\label{v_e_def}
v_e= k(\omega,e) \, e
\end{equation}
for every $\omega\in\Omega$ and $e\in \mathcal U_d$.
Let $\big\{0=x_0,x_1,\ldots,x_k=k(\omega,e)\,e\big\}$ be an $\omega$-open path from $0$ to $k(\omega,e)\,e$. 
For every $G(\cdot,\cdot)$ that satisfies \eqref{closedloop}, we again write 
\begin{equation}\label{V_ve}
V(\omega,v_e)=V_G(\omega, v_e)= \sum_{i=0}^{k-1} G(\tau_{x_i}\omega,x_{i+1}-x_i).
\end{equation}
Again, the choice of the path doesn't influence $V(\omega,v_e)$. We then say that $V=V_G$ satisfies the {\it{induced zero mean property}} by requiring that for every $e\in\mathcal U_d$,
\begin{equation}\label{meanzero}
\E_0\big[V(\cdot,v_e)\big]=0.
\end{equation}
\end{itemize}

\subsection{Sub-linear growth of the correctors at infinity}\label{subsec-sublinear}

This section is devoted to the proof of the following important property of functions in the class $\mathcal G_\infty$.
 \begin{theorem}[Sub-linear growth at infinity on the cluster]\label{sublinearthm}
 For every $G\in \mathcal G_\infty$, $V=V_G$ has at most sub-linear growth at infinity on the infinite cluster $\P_0$- almost surely,. In other words,
  $$
  \lim_{n\to\infty} \max_{\heap{x\in \mathcal C_\infty}{|x|\leq n}}  \frac {|V(\omega,x)|}n =0.
  $$
  \end{theorem}

Before we present the proof of Theorem \ref{sublinearthm}, which is carried out at the end of this section, we need some important estimates related to the geometry of the 
 infinite percolation cluster $\mathcal C_\infty$ presented in the following two lemmas. Lemma \ref{chemdist} gives a precise bound on the shortest distance of two points 
 in the infinite cluster (the {\it{chemical distance}}) and Lemma \ref{lemma-ell} gives an exponential tail bound on the graph distance between the origin and the 
the first arrival  $v_e=k(\omega,e)e$ (recall \eqref{v_e_def}), 
of the cluster on the positive part of each of the coordinate directions. Both lemmas are well-known in the literature covering i.i.d. Bernoulli percolation and the site percolation model discussed in Section \ref{sec-intro-site}.
 For the random cluster model, contents of these two results are new, to the best of our knowledge. Apart from the proof of Theorem \ref{sublinearthm}, both lemmas will be helpful in carrying out our variational analysis 
 in Section \ref{sec-proof-ldp} (see the proof of Lemma \ref{lemma-last-last}).

  
%

We first turn to the following estimate 
on the {\it{chemical distance}} $\d_{\mathrm{ch}}(x,y)=\d_{\mathrm{ch}}(\omega;\,x,y)$ of two points  $x,y \in \mathcal C_\infty$,  
which is defined to be the minimal length of an $\omega$-open
path connecting $x$ and $y$ in the configuration $\omega\in \Omega_0$. 
The following result, originally proved by Antal and Pisztora (\cite{AP96}) for supercritical i.i.d. Bernoulli percolation, asserts that 
the chemical distance of two points in the cluster is comparable to their Euclidean distance. 
\begin{lemma}\label{chemdist}
Assumption 3 holds for the percolation models introduced in Section \ref{sec-intro-bond} and Section \ref{sec-intro-site}.
In particular, let us fixe $\delta>0$. Then there exists a constant $\rho=\rho(p,d)$ such that, $\P_0$- almost surely, for every $n$ large enough and points $x,y\in \mathcal C_\infty$ with $|x|<n, |y|< n$ and $\delta n/2\leq|x-y|< \delta n$, we have
$\d_{\mathrm{ch}}(x,y) < \rho \delta n$.
\end{lemma}

  \begin{proof}  
  
 For i.i.d. Bernoulli bond and site percolation model (recall Section \ref{sec-intro-bond}), the statement of this lemma
follows from the classical estimate of Antal-Pisztora (Theorem 1.1, \cite{AP96})). 
 

 


We now prove the lemma for the supercritical random-cluster model. Recall that we 
assume that $p > \widehat{p}_{c}(q)$. 
 For every $r\geq 0$, we define a box $B_{0}(r) := [-r, r]^{d}$ and set, for any $z \in\Z^d$, $N\in\N$, 
$$
B_{z}(N) = \tau_{(2N+1)z} B_{0}(5N/4) 
$$ 
Here $\tau_{z}$ is the transformation on $\mathbb{Z}^{d}$ defined by $\tau_{z}(x) = z + x$.
We define $R_{z}^{(N)}$ to be the event in $\{0,1\}^{\mathbb B_d}$ satisfying the following three conditions:  
\begin{itemize}
\item There exists a {\it{unique crossing open cluster}} for $B_{z}(N)$. In other words,  
there is a connected subset $\mathcal{C}$ of an open cluster such that it is contained in $B_{z}(N)$, and, 
for all $d$ directions there is a path in $\mathcal{C}$ connecting the left face and the right face of $B_{z}(N)$. 
\item  The cluster in the above requirement intersects all boxes with diameter larger than $N/10$.
\item All open clusters with diameter larger than $N/10$ are connected in $B_{z}(N)$. 
\end{itemize}
Recall the measures $\mathbb{P}_{\Lambda, p,q}^{\ssup \xi}$ and $\mathbb{P}_{p,q}^{\ssup b}$ corresponding to the random cluster model. 
Then, under the map
$$
\phi_{N} : \{0,1\}^{\mathbb B_d} \to \{0,1\}^{\mathbb{Z}^d} \qquad (\phi_{N}\omega)_{z} = \1_{R_{z}^{(N)}}(\omega)\quad\forall\, z \in \mathbb{Z}^{d},
$$
we let 
$$
\mathbb{P}_{p,q,N}^{\ssup b}= \mathbb{P}_{p,q}^{\ssup b} \, \circ \,  \phi_N^{-1}
$$
to be the image measure of $\mathbb{P}_{p,q}^{\ssup b}$. 
By \cite[Theorem 3.1]{P96} for $d \ge 3$ and \cite[Theorem 9]{CM04} for $d = 2$, 
we see that  
there exist constants $c_{1}^{\prime}, c_{2}^{\prime} > 0$ (depending only on $d$, $p$ and $q$), 
such that for every $N \ge 1$ and $i \in \mathbb{Z}^{d}$,   
$$
\sup_{\xi \in \Omega}\, \mathbb{P}_{B_{z}(N), p, q}^{\ssup \xi}\,\, \bigg[ \big(R_{z}^{(N)}\big)^{c}\bigg]
\le c_{1}^{\prime}\,\, \e^{-c_{2}^{\prime} N}. 
$$

Let $Y_{z} : \{0,1\}^{\mathbb{Z}^d} \to \{0,1\}$ be the projection mapping to the coordinate $z \in \mathbb{Z}^{d}$.  
By using the DLR property for the random-cluster model (\cite[Section 4.4]{G06}), for both boundary conditions $b$,
$$
 \lim_{N \to \infty}  \sup_{z \in \mathbb{Z}^{d}} 
\mathrm{ ess.sup }\, \, 
\mathbb{P}_{p,q,N}^{\ssup b}\,\, \bigg[Y_{z} = 0  \bigg | \sigma\bigg(Y_{x} :  |x-z|_{\infty} \ge 2\bigg)\bigg] 
= 0. 
$$
By using \cite[Theorem 1.3]{LSS97},
we see that there exists a function $\overline{p}(\cdot): \N\to [0,1)$  
such that $\overline{p}(N) \to 1$ as $N \to \infty$
and the Bernoulli product measure 
$$
\mathbb{P}^{\star}_{\overline{p}(N)}=\big(\overline p(N)\,\,\delta_1\,\,+\,\,\big(1-\overline p(N)\,\big)\delta_0\big)^{\Z^d}
$$ 
on $\{0,1\}^{\mathbb{Z}^d}$ with parameter $\overline p(N)$ is dominated by $\mathbb{P}_{p, q, N}^{\ssup b}$ for each $N$, i.e., 
for every increasing event $A$, 
$$
\mathbb{P}^{\star}_{\overline{p}(N)}(A) \leq \mathbb{P}_{p, q, N}^{\ssup b}(A).
$$
 
Given the above estimate, we can now repeat the arguments in (p. 1047, \cite{AP96}) to conclude that

\begin{equation}\label{est-chemdist}
\P_0\bigg\{\d_{\mathrm{ch}}(x,y)> \rho |x-y|, \ x, y \in \mathcal C_\infty\bigg\} \leq \e^{-c|x-y|}.
\end{equation}
For some suitably chosen $\rho>0$. Borel-Cantelli lemma and our assumption that $|x-y|\geq \delta n/2$ now conclude the proof of Lemma \ref{chemdist} for random cluster models. 

For the site percolation models (i.e., random interlacements, its vacant set and the level sets of the Gaussian free field) introduced in Section \ref{sec-intro-site}, Lemma \ref{chemdist} 
follows from the estimate 
$$
\P_0\bigg\{\d_{\mathrm{ch}}(x, y) \geq \rho\, |x-y|, \,x, y \in \mathcal C_\infty\bigg\}\,\leq c_{1} \,\,\e^{ - c_1 \,\,(\log |x-y|)^{1+c_{2}}},
$$
for constants $c_1,c_2>0$ and every $x\in\Z^d$. This statement and its proof can be found in \cite[Theorem 1.3]{DRS14}.
The above estimate and the assumption $|x-y| \ge \delta n / 2$ concludes the proof of Lemma \ref{chemdist}.
\end{proof}

For every $e\in\mathcal U_d$, we recall that $v_e=k(\omega,e)e$, recall \eqref{def-n} and \eqref{v_e_def}. Let 
$\ell=\ell(\omega)$ denote the shortest path distance from $0$ to $v_e$. Then we have the following tail estimate on $\ell$: 

\begin{lemma}\label{lemma-ell}
Assumption 4 holds for the percolation models introduced in Section \ref{sec-intro-bond} and Section \ref{sec-intro-site}. In particular, for some constant $c_1,c_2>0$, 
$$
\P_0\big\{\ell>n\big\}\leq c_1\,\e^{-c_2 n}.
$$
\end{lemma}


\begin{proof}
Lemma \ref{lemma-ell} follows from (\cite[Lemma 4.3]{BB07}) for i.i.d. Bernoulli bond and site percolations, 
and from \cite[Section 5]{PRS15} for the site percolation models appearing in Section \ref{sec-intro-site}. 

We turn to the requisite estimate corresponding to the random cluster model defined in Section \ref{sec-intro-bond}.
Let us first handle the case $d \ge 3$ and recall the definition of slab-critical probability $\widehat p_c(q)$ from
\eqref{slab-d3} and recall that we assume $p>\widehat p_c(q)$. Then we can take a large number $L$ so that $p > \widehat p_c(q, L)$ and $[0, L-1]\times\Z^{d-1}$ contains
an infinite cluster, which is a subset of the unique infinite cluster $\mathcal C_\infty$. 
For every $e\in \mathcal U_d$, we recall the definition of $k(\omega,e)$ from \eqref{def-n} and note that we write $v_e=k(\omega,e)e$. 
Also, by symmetry of the random-cluster measure, we can assume $e = e_1$ without loss of generality. Then, 
\begin{equation}\label{lemma-ell-1}
\bigg\{|v_e|\geq  L n;\, 0\in\mathcal C_\infty\bigg\} \subset \bigcap_{i=1}^n\, \tau_{i Le}(A_L^c),
\end{equation}
where 
$$
A_L := \bigcup_{j = 1}^{L} \bigg\{0 \leftrightarrow je \textup{ in } [0,L-1] \times \Z^{d-1} \bigg\}.  
$$
We also define, for $m \ge 1$, 
$$
A_{L, m} := \bigcup_{j = 1}^{L} \{0 \leftrightarrow je \textup{ in } S(L,m) \}.  
$$
Then $\{A_L\}_{L>0}$ and $\{A_{L,m}\}_{m\geq 1}$ are increasing events. 
Then, by the DLR property of the random-cluster measure with the free boundary condition (\cite[Definition 4.29]{G06}), and the extremality of the random-cluster measure with the free boundary condition (\cite[Lemma 4.14]{G06}),
$$
\mathbb{P}_{p,q}^{\ssup 0}\left(\bigcap_{i=1}^n\, \tau_{i Le}(A_{L,m}^c)\right) \le \left(1-\mathbb{P}_{S(L,m), p,q}^{\ssup 0}(A_{L,m})\right)^n 
$$
Since $p> \widehat p_c(q, L)$, 
$$
\liminf_{m \to \infty} \mathbb{P}_{S(L,m), p,q}^{\ssup 0}(A_{L,m}) > 0. 
$$
Hence for some $0 < a(L) < 1$, 
\begin{equation}\label{lemma-ell-2}
\mathbb{P}_{p,q}^{(0)}\left(\bigcap_{i=1}^n\, \tau_{i Le}(A_{L}^c)\right) = \lim_{m \to \infty} \mathbb{P}_{p,q}^{(0)}\left(\bigcap_{i=1}^n\, \tau_{i Le}(A_{L,m}^c)\right) \le a(L)^n 
\end{equation} 
Then \eqref{lemma-ell-1} and the above estimate imply that for $d\geq 3$, $\P_0\big\{|v_e|\geq  n\big\}$ decays exponentially in $n$. 
To prove this statement 
in $d=2$, we again recall the definition of the slab-critical probability $\widehat p_c(q)$ and note that  $p>\widehat p_c(q)$. In this regime,
we have exponential decay of truncated connectivity (see \cite[Theorem 5.108 and the following paragraph]{G06}). In other words, if $\mathcal C$ denotes an open cluster at the origin, 
then
\begin{equation}\label{lemma-ell-4}
\lim_{n\to\infty}\frac 1n \log \P_{p,q}^{\ssup 0}\bigg\{\big|\mathcal C\big| \geq n^2\,\,;\big|\mathcal C\big|<\infty\bigg\}<0.
\end{equation}
In this super-critical regime $p>p_c(q)$, we also have exponential decay of dual connectivity (see \cite[Theorems 1 and 2]{BD12}).
In other words, in the dual random cluster model in $d=2$, the probability for two points $x$ and $y$ 
to be connected by a path decays exponentially fast with respect to the distance between $x$ and $y$.
We remark that in this case the infinite volume limits of the random-cluster measures with free or wired boundaries are identical if $p > p_c (q)$. 
See \cite[Theorem 4.63, (4.36) and Theorem 6.17]{G06}. 
Futhermore, in this case, \cite{BD12} shows $\widehat{p}_c(q) = p_c(q)$ for every $q \ge 1$.

To show that $\P_0\big\{|v_e|\geq n\big\}$ decays exponentially in $n$ in $d=2$,
we now let $B_n$ to be the box $\{1,\dots,n\}\times\{1,\dots, n\}$. Then on the event $\{|v_e|\geq n, \,\,; 0\in\mathcal C_\infty\}$, none of the boundary sites $\{je: j = 1,...,n\}$ are in 
$\mathcal C_\infty$. Hence, either at least one of these sites is in a finite component of size larger than $n$ or there exists a dual crossing of $B_n$ in the direction of $e$.
The probabilities of both these events are exponentially small in $n$ by \eqref{lemma-ell-4} and the exponential decay of dual connectivity. Hence  $\P_0\big\{|v_e|\geq n\big\}$ decays exponentially in $n$ 
for $d\geq 2$.

To conclude the proof of Lemma \ref{lemma-ell}, we note that 
for  every $\eps>0$, 
\begin{equation}\label{lemaa-ell-3}
\big\{\ell>n\big\}  = \bigcup_{j=1}^{\lceil\eps n\rceil} \bigg\{\d_{\mathrm{ch}}(0,je)>n\,;\,\,0,je\in  \mathcal C_\infty\bigg\}\,\, \bigcup \bigg\{|v_e| \geq \lceil\eps n\rceil\bigg\}.
\end{equation}
Since $\P_0$-probabilities of the events in the first union are exponentially small by the uniform estimate \eqref{est-chemdist} on the chemical distance $\d_{\mathrm{ch}}(0, je)$,
and $\P_0\big\{|v_e|\geq n\big\}$ also decays exponentially in $n$, we now invoke union of events bound and absorb the linear factor coming from the number of events 
in the exponential bound and end up with the proof of Lemma \ref{lemma-ell}.
\end{proof}
\begin{lemma}\label{lemma-FKG}
The FKG inequality (i.e., Assumption 5) holds for percolation models introduced in Section \ref{sec-intro-bond} and Section \ref{sec-intro-site}.
\end{lemma}
\begin{proof}
For the proof of The FKG inequality we refer to \cite[Theorem 4.17]{G06} for the random cluster model, to \cite{T09} for the random interlacement and the vacant set of it, 
and to \cite[Remark 1.4]{R15} for the level sets of Gaussian free fields. 
\end{proof}

Before we turn to the proof of Theorem \ref{sublinearthm}, we will need another technical fact. Note that, by Birkhoff's ergodic theorem, 
$$
\lim_{n\to\infty}\frac 1 {(2n+1)^d} \sum_{|x|\leq n} \, \1\big\{x\in \mathcal C_\infty\big\}=\theta(p) \qquad\P_0-\,\mbox{ a.s.}
$$
where $\theta(p)=\P(0\leftrightarrow\infty)$ is the percolation density. We will need 
a stronger version of the above result and its argument will use the one dimensional pointwise ergodic theorem and an
induction argument on the dimension. 
We will prove this result for every discrete point process (i.e. a shift
invariant ergodic random subset of $\mathbb Z^d$). For our case, we will take our infinite cluster $\mathcal C_\infty$ to be the point process.

\begin{lemma}\label{lemma-ergodic}
Let $\mathcal P$ be a discrete point process in $d$ dimensions, and let
$C^{\ssup d}=[a_1,b_1]\times\cdots\times[a_d,b_d]$ be a cube in $\mathbb R^d$.
Then for almost every $\omega$, 
$$
\lim_{n\to\infty}\frac{|\mathcal P\cap n\,\,C^{\ssup d}|} {|n\,\,C^{\ssup d}|} = \Theta
$$
where $\Theta=\P(0\in\mathcal P)$ is the density of $\mathcal P$.
\end{lemma}

\begin{proof}
We will prove the Lemma by induction on the dimension $d$. For $d=1$, the Lemma follows directly from the pointwise
ergodic theorem, when we subtract the sum in $[a_1n]$ from that in
$[b_1n]$.

 We now assume that the statement holds for dimension $d-1$.
We fix $\eps>0$ and $K$ say that $n\in\mathbb Z$ is good if for every $k>K$,
$$
\left|
\frac{|\mathcal P\cap \{n\}\times
k([a_2,b_2]\times\cdots\times[a_d,b_d])|}
{|k([a_2,b_2]\times\cdots\times[a_d,b_d])|}
- \Theta
\right|
<\eps.
$$

Note that if $K$ is large enough, then by the induction hypothesis the
probability that $0$ is good is greater than $1-\eps$. So by the one
dimensional result, a.s. for all $n$ large enough, proportion larger than
$1-\eps$ of the numbers in $[a_1n,b_1n]$ are good, and the statement of the lemma
follows.
\end{proof}

We now turn to the proof of Theorem \ref{sublinearthm}. 
  
\noindent {\bf{Proof of Theorem \ref{sublinearthm}:}} Let us fix every $G\in \mathcal G_\infty$ and for every nearest neighbor occupied path $0=x_0,\dots,x_n=x$ in $\mathcal C_\infty$, 
let $V(\omega,x)=V_G(\omega,x)=\sum_{j=0}^{n-1} G(\tau_{x_j},x_{j+1}-x_j)$ as defined in \eqref{Psidef}. Recall that we have to show 
\begin{equation}\label{eq0}
\lim_{n\to\infty} \max_{\heap{x\in\mathcal C_\infty}{|x|\leq n}} \,\,\frac{|V(\omega,x)|}n =0\qquad\P_0-\,\,\mathrm{a.s.}
\end{equation}
Let us first make an observation based on the facts proved in Lemma \ref{chemdist} and Lemma \ref{lemma-ell} . Indeed, with of $V(\omega,v_e)$ defined in \eqref{V_ve}, Lemma \ref{lemma-ell}  and our 
uniform bound assumption \eqref{unifbound} imply that $\E_0 [|V(\omega,v_e)|] <\infty$. Furthermore, $\E_0[V(\omega, v_e)]=0$ by our induced mean-zero assumption \eqref{meanzero}. If we now write $F(\omega)= V(\omega,v_e)$ and recall that  $n_{k+1}(\omega,e)=n_k(\sigma_e\omega,e)$ from \eqref{def-nk}, then $V(\omega,v_e)=V(\omega, n_k(\omega,e) \, e)= \sum_{j=0}^{k-1} F \, \circ \,\sigma^j_e(\omega)$. Since 
the induced shift $\sigma_e:\Omega_0\rightarrow\Omega_0$ is measure-preserving and ergodic, by Birkhoff's ergodic theorem, 
\begin{equation}\label{eq00}
\lim_{k\to\infty} \frac 1 k V(\omega, n_k(\omega,e) \, e)=0 \qquad\P_0-\,\,\mathrm{a.s.}
\end{equation}
We now fix an arbitrary $\eps>0$. We claim that
\begin{equation}\label{eq01}
 \lim_{n\to\infty} \frac 1{n^d} \sum_{\heap {x\in \mathcal C_\infty}{|x|\leq n}} \, \1_{\big\{|V(x,\omega)|>\eps n\big\}} =0 \qquad\P_0-\,\,\mathrm{a.s.}
 \end{equation}
Actually \eqref{eq00} forms the core of the argument for the proof of \eqref{eq01}. Indeed, given \eqref{eq00}, the proof of the claim \eqref{eq01} for all the percolation models
including long-range correlations introduced in Section \ref{sec-intro-bond} and Section \ref{sec-intro-site}, now closely follows the proof of  \cite[Theorem 5.4]{BB07} deduced for i.i.d. Bernoulli percolation. 
In fact, the crucial fact \cite[(5.28)]{BB07} can be proved using the FKG inequality. 
Recall that the FKG inequality asserts that for two increasing events $A$ and $B$ (i.e, events that are preserved by addition of open edges), $\P(A\cap B)\geq \P(A)\P(B)$.
Hence, based on the assertion \eqref{eq00} we have just proved and using Lemma \ref{lemma-FKG}, 
we can repeat the arguments of \cite[Theorem 5.4]{BB07} to prove the assertion \cite[(5.28)]{BB07} therein and thus deduce \eqref{eq01}. 
Then, for every arbitrary $\eps>0$, \eqref{eq01} in particular implies that, for $n$ large enough, 
\begin{equation}\label{eq1}
\sum_{\heap {x\in \mathcal C_\infty}{|x|\leq n}} \, \1_{\big\{|V(x,\omega)|>\eps n\big\}} < \eps n^d \qquad\P_0-\,\,\mathrm{a.s.}
\end{equation}
Let us make another observation based on Lemma \ref{chemdist}. Recall that $\theta(p)>0$ denotes the percolation density, i.e., $\theta(p)$ is the probability that $0$ is in the infinite open cluster $\mathcal C_\infty$.
Also, for every arbitrary $\eps>0$ as before, let us set 
\begin{equation}\label{delta-def}
\delta= \frac 12 \bigg(\frac{4\eps}{\theta(p)}\bigg)^{\frac 1d}.
\end{equation}
Then by Lemma \ref{chemdist}, for every $x,y\in \mathcal C_\infty$ with $|x|<n$, $|y|<n$ and $\delta n /2\leq |x-y| < \delta n$, 
\begin{equation}\label{eq2}
\d_{\mbox{ch}}(x,y) < \rho \delta n.
\end{equation}
Finally, let us recall Lemma \ref{lemma-ergodic}. 
Hence for every fixed $\delta>0$, for every $n$ large enough and $\P_0$-almost surely, in a ball of radius $\delta n$ in $\mathcal C_\infty \cap [-n,n]^d$ 
there are at least $\delta^d (2n)^d \,\frac {\theta} 2$ points in $\mathcal C_\infty$ (Lemma \ref{lemma-ergodic} suffices for the above statement because we take the infinite cluster $\mathcal C_\infty$ as our point
process, and we use Lemma \ref{lemma-ergodic} for finitely many cubes $C^{\ssup d}$).

Then for our choice of $\delta$ as required in \eqref{delta-def}, 
\begin{equation}\label{eq3}
\begin{aligned}
\#\big\{\mbox{points in a box of radius}\,\, \delta n\,\,\mbox{in}\,\,[-n,n]^d\,\,\mbox{in}\,\, \mathcal C_\infty\} 
> 2\eps n^d.
\end{aligned}
\end{equation}
Given \eqref{eq1} and \eqref{eq3}, we now claim that, for large enough $n$ and every $x\in [-n,n]^d$, there exists $y\in  [-n,n]^d \cap \mathcal C_\infty$ so that $|y-x|< \delta n$ and 
$$
|V(\omega,y)| \leq \eps n\qquad\P_0-\,\,\mathrm{a.s.}
$$
Indeed, by \eqref{eq1} there are at most $\eps n^d$ points $z\in [-n,n]^d$ such that $|V(\omega,z)| \geq \eps n$ and by \eqref{eq3}, there are at least $2\eps n^d$ points in $B_{n\delta}(x)\cap \mathcal C_\infty$. Hence, we have at least one point $y\in  [-n,n]^d \cap \mathcal C_\infty$ such that   $\delta n /2\leq |y-x|< \delta n$ and $|V(\omega,y)| \leq \eps n$, 
$\P_0$- almost surely.

Let us now prove \eqref{eq0}. Recall the definition of $V$ from \eqref{Psidef}. Then, by \eqref{eq2},
$$
\begin{aligned}
\big| V(\omega,x) - V(\omega,y)\big| &\leq \d_{\mbox{ch}}(x,y) \,\mbox{ess} \sup_{\omega- \P_0} G(\omega, x)\\
&\leq \rho \delta n A,
\end{aligned}
$$
 for some $A< \infty$, recall \eqref{unifbound}. Since $|V(\omega,y)| \leq \eps n$, then $\P_0$- almost surely, 
$$
\begin{aligned}
|V(\omega,x)| &\leq |V(\omega,y)|+ \rho \delta n A \\
& \leq\eps n+ \rho \delta n A.
\end{aligned}
$$
Since $\eps>0$ is arbitrary and $\delta\to 0$ as $\eps\to 0$ according to \eqref{delta-def}, Theorem \ref{sublinearthm} is proved.
\qed

We have an immediate corollary to Theorem \ref{sublinearthm}.
\begin{cor}\label{cor-sublinear}
Let $G\in \mathcal G_\infty$. For every $\eps>0$,  there exists $c_\eps=c_\eps(\omega)$ so that, for every sequence of points $(x_k)_{k=0}^n$ on $\mathcal C_\infty$ with $x_0=0$ and $|x_{k+1}-x_{k}|=1$,
$$
\bigg|  \sum_{k=0}^{n-1} G(\tau_{x_k} \omega, x_{k+1}-x_k)\bigg| \leq c_\eps+n\eps.
$$
In particular, 
\begin{equation}\label{ub2}
 \sum_{k=0}^{n-1} G(\tau_{x_k} \omega, x_{k+1}-x_k) \geq -c_\eps- n\eps.
\end{equation}
\end{cor}

\subsection{Proof of the upper bound for Theorem \ref{thmmomgen}.}\label{subsec-ub}
We now prove the upper bound in Theorem \ref{thmmomgen}
using the sub-linear growth property of gradient functions established in Theorem \ref{sublinearthm} and Corollary \ref{cor-sublinear}.
\begin{lemma}[The upper bound]\label{ub}
For $\P_0$- almost every $\omega\in \Omega_0$,  
$$
\limsup_{n\to\infty} \frac 1n\log E_{0}^{\pi,\omega} \bigg\{ \exp \big\{\sum_{k=0}^{n-1} f\big(\tau_{X_k}\omega, X_{k+1}-X_k\big)\big\}\bigg\} \leq \inf_{G\in \mathcal G_\infty}\Lambda(f, G),
$$
where
\begin{equation}\label{UfG}
\Lambda(f, G)=\mathrm{ess}\sup_{\P_0} \bigg(\log \sum_{e} \1_{\{\omega_e=1\}}\pi_\omega(0,e) \exp\big\{f(\omega,e)+ G(\omega,e)\big\}\bigg).
\end{equation}
\end{lemma}

\begin{proof}
Fix $G\in \mathcal G_\infty$. By the definition of the Markov chain $P^{\pi,\omega}_0$ we have $\P_0$-a.s., 
$$
\begin{aligned}
&E^{\pi,\omega}_0\bigg\{\exp\bigg\{ f(\tau_{X_{k}}\omega,X_{k+1}-X_{k})+G(\tau_{X_{k+1}}\omega,X_{k+1}-X_{k})\bigg\}\bigg|X_{k}\bigg\} \\
&= \sum_{|e|=1} \pi_\omega\big(X_{k},X_{k}+e\big) \e^{f(\tau_{X_{k}}\omega,e)+ G(\tau_{X_{k}}\omega,e)}\\
&=\sum_{|e|=1}\1_{\{(\tau_{X_{k}}\omega)(e)=1\}} \pi_\omega\big(X_{k},X_{k}+e\big) \e^{f(\tau_{X_{k}}\omega,e)+ G(\tau_{X_{k}}\omega,e)}\\
&\leq \e^{\Lambda(f, G)},
\end{aligned}
$$
where the uniform upper bound follows from \eqref{UfG}.

Invoking the Markov property and successive conditioning, we have

\begin{equation}\label{ub1}
E^{\pi,\omega}_0\bigg\{\exp\bigg\{\sum_{k=0}^{n-1} \bigg(f(\tau_{X_k}\omega,X_{k+1}-X_{k})+G(\tau_{X_k}\omega,X_{k+1}-X_{k})\bigg)\bigg\}\bigg\} \leq \e^{n\Lambda(f, G)}.
\end{equation}
We now recall Corollary \ref{cor-sublinear} and plug in the lower bound \eqref{ub2} in \eqref{ub1}. Then if we divide both sides by $n$, take logarithm and pass to $\limsup_{n\to\infty}$, we obtain the upper bound
$$
\limsup_{n\to\infty} \frac 1n\log E^{\pi,\omega}_0 \bigg\{ \exp \big\{\sum_{k=0}^{n-1} f\big(\tau_{X_k}\omega, X_{k+1}-X_k\big)\big\}\bigg\} \leq \Lambda(f, G)+ \eps.
$$
Passing to $\eps\to 0$ and subsequently taking $\inf_{G\in \mathcal G_\infty}$ we finish the proof of the lemma.
\end{proof}

\section{Equivalence of bounds: Min-max Theorems based on entropic coercivity}\label{sec-proof-ldp}

In this section we turn to the proof of the crucial fact that the lower bound obtained from Corollary \ref{corlb} and the upper bound from Lemma \ref{ub} indeed match. The following theorem holds the key argument of our analysis and will also finish the proof of Theorem \ref{thmmomgen}. Recall the lower bound variational formula $\overline H(f)$ from \eqref{thm-lbub2a}, and the upper bound variational formula $\Lambda(f,G)$ from \eqref{UfG}.


\begin{theorem}[Equivalence of bounds]\label{thm-lbub}
For every continuous and bounded function $f$ on $\Omega_0\times\mathcal U_d$, 
$$
\begin{aligned}
\overline H(f)&= \inf_{G\in\mathcal G_\infty} \,\mathrm{ess}\,\sup_{\P_0}\,\bigg(\log\sum_{e\in\mathcal U_d} \1_{\omega(e)=1}\,\pi_\omega(0,e)\, \exp\big\{f(\omega,e)+ G(\omega,e)\big\}\bigg) \\
&=\inf_{G\in\mathcal G_\infty} \Lambda(f,G)
\end{aligned}
$$
\end{theorem}
We will prove Theorem \ref{thm-lbub} in several steps. Note that we already know that $\overline H(f)\leq \inf_{G\in\mathcal G_\infty} \Lambda(f,G)$ and it remains the prove the inequality in the opposite direction.
The first step is to invoke a min-max argument to exchange the order of $\sup_{\widetilde\pi}$ and $\inf_g$ in \eqref{thm-lbub2}, and subsequently solve the maximization problem in $\widetilde\pi$. The resulting assertion is
\begin{lemma}[Entropic coercivity in $\widetilde\pi$]\label{thm-lbub-lemma1}
For every continuous and bounded function $f$ on $\Omega_0\times\mathcal U_d$, 
$$
\overline H(f)=\sup_{\phi}\, \inf_{g} \int \d\P_0(\omega)\,\phi(\omega) L(g,\omega)
$$
where
\begin{equation}\label{thm-lbub5}
L(g, \omega)=L_f(g,\omega)=\log\bigg(\sum_{e\in\mathcal U_d}\pi_\omega(0,e)\exp\big\{f(\omega,e)+ g(\omega)-g(\tau_e\omega)\big\}\bigg).
\end{equation}
\end{lemma}
\begin{proof}
Let us rewrite \eqref{thm-lbub2a} as
\begin{equation}\label{thm-lbub2}
\begin{aligned}
\overline H(f)&= \sup_{\phi}\, \sup_{\tilde\pi\in\widetilde\Pi}\,\inf_{g} \bigg[\int \d\P_0(\omega)\,\phi(\omega) \bigg\{
\sum_{e\in \mathcal U_d}  \tilde\pi(\omega,e) \big(F(\omega,e) -\log\widetilde\pi(\omega,e)\big)\bigg\}\bigg],
\end{aligned}
 \end{equation}
where
\begin{equation}\label{def-F}
F(\omega,e)=F(\pi,f,g,\omega,e)=f(\omega,e) + \log {\pi_\omega(0,e)}
+\big(g(\omega)-g(\tau_e\omega)\big).
\end{equation}
and the infimum is being taken over bounded and measurable $g$. For every $\widetilde\pi$ and $g$, let us write the functional 
\begin{equation}\label{def-frakF}
\begin{aligned}
\mathfrak F(\widetilde\pi,g)&= \int \d\P_0(\omega)\,\phi(\omega) \bigg\{
\sum_{e\in \mathcal U_d}  \tilde\pi(\omega,e) \big[f(\omega,e) + \log {\pi_\omega(0,e)}
+\big(g(\omega)-g(\tau_e\omega)\big) -\log\widetilde\pi(\omega,e)\big]\bigg\} \\
&=\int \d\P_0(\omega)\,\phi(\omega) \bigg\{
\sum_{e\in \mathcal U_d}  \tilde\pi(\omega,e) \big[F(\omega,e) -\log\widetilde\pi(\omega,e)\big]\bigg\}
\end{aligned}
\end{equation}
with $F(\omega,e)$ defined in \eqref{def-F} (recall \eqref{thm-lbub2}). Let us fix an arbitrary density $\phi$ (i.e., $\phi\geq 0$ and $\E_0\phi=1$). First we would like to show that, 
\begin{equation}\label{thm-lbub-minmax1}
\sup_{\tilde\pi\in\widetilde\Pi}\,\,\inf_{g} \mathfrak F(\widetilde\pi,g)= \inf_g\,\,\sup_{\tilde\pi\in\widetilde\Pi} \mathfrak F(\widetilde\pi,g).
\end{equation}
This requires the following coercivity argument. 
Note that,  corresponding to each $\widetilde\pi\in \widetilde\Pi$,  we have the entropy functional 
$$
\mathrm{Ent}(\mu_{\widetilde\pi})=\int\sum_e \widetilde\pi(\omega,e)\log\widetilde\pi(\omega,e) \phi(\omega)\d\P_0(\omega)
$$
for the probability measure $\d\mu_{\widetilde\pi}(\omega,e)= \widetilde\pi(\omega,e) \big(\phi(\omega)\d\P_0(\omega)\big) \in \Mcal_1(\Omega_0\times \mathcal U_d)$.
Then for every fixed $\phi$, the map $\widetilde\pi\mapsto \mathrm{Ent}(\mu_{\widetilde\pi})$ is convex, lower semi-continuous and has weakly compact sub-level sets (i.e., for every $a\in \R$, the set $\{\widetilde\pi\in\widetilde\Pi\colon \mathrm{Ent}(\mu_{\widetilde\pi}) \leq a\}$ is weakly compact). Furthermore, for every probability density $\phi$, every continuous and bounded function $f$ on $\Omega_0\times\mathcal U_d$ and bounded measurable function $g$, and for every $\widetilde\pi\in \widetilde\Pi$, 
$$
\begin{aligned}
\int \phi(\omega) \sum_{e\in\mathcal U_d} \widetilde \pi(\omega,e) F(\omega,e)\d\P_0&= \int \phi(\omega) \sum_{e\in\mathcal U_d} \widetilde \pi(\omega,e) \big[f(\omega,e) + \log {\pi_\omega(0,e)}
+\big(g(\omega)-g(\tau_e\omega)\big] \d\P_0
\\
&\leq \big(\|f\|_\infty+2\|g\|_\infty\big) \,\int \phi(\omega) \sum_{e\in\mathcal U_d} \widetilde \pi(\omega,e)\d\P_0\\
&=\|f\|_\infty+2\|g\|_\infty:=C<\infty.
\end{aligned}
$$
We conclude that for every bounded and measurable $g$, the map $\widetilde\pi\mapsto \mathfrak F(\widetilde\pi,g)$ is concave, weakly upper-semicontinuous and has weakly compact super-level sets  $\{\widetilde\pi\colon\, \mathfrak F(\widetilde\pi,g)\geq a\}$ for every $a\in \R$. Furthermore, for every $\widetilde\pi\in \widetilde\Pi$, the map $g\mapsto \mathfrak F(\widetilde\pi,g)$ is linear and continuous in $g$. Then, in view of {\it{Von-Neumann's min-max theorem}} (p. 319, \cite{AE84}), the equality \eqref{thm-lbub-minmax1} holds. Hence, 
 \begin{equation}\label{thm-lbub3}
 \begin{aligned}
 \overline H(f) = \sup_{\phi}\, \inf_{g}\, \sup_{\widetilde\pi}\, \bigg[\int \d\P_0(\omega)\,\phi(\omega) \bigg\{
&\sum_{e\in \mathcal U_d}  \tilde\pi(\omega,e)\big[F(\omega,e)-\log \widetilde\pi(\omega,e)\big]\bigg\}\bigg]
\end{aligned}
 \end{equation}
Since the integrand above depends only locally in $\widetilde\pi$, we can bring the $\sup_{\widetilde\pi\in\widetilde\Pi}$ inside the integral, and solve the 
variational problem 
$$
\sup_{\widetilde\pi}\sum_{e\in \mathcal U_d}  \tilde\pi(\omega,e)\big[F(\omega,e)-\log \widetilde\pi(\omega,e)\big]
$$
subject to the Lagrange multiplier constraint $\sum_e \widetilde\pi(\cdot,e)=1$. The maximizer is 
$$
\widetilde\pi(\cdot,e)=\frac{\exp[F(\omega,e)]}{\sum_{e\in\mathcal U_d}\exp[F(\omega,e)]},
$$ 
and if we plug in this value in \eqref{thm-lbub3} and recall the definition of $F(\omega,e)$ from \eqref{def-F},
then \eqref{thm-lbub3} leads us to
\begin{equation}\label{thm-lbub4}
\begin{aligned}
\overline H(f) &=  \sup_{\phi}\, \inf_{g} \bigg[\int \d\P_0(\omega)\,\phi(\omega) \log\bigg(\sum_{e\in\mathcal U_d}\pi_\omega(0,e)\exp\big\{f(\omega,e)+ g(\omega)-g(\tau_e\omega)\big\}\bigg)\bigg] \\
&=\sup_{\phi}\, \inf_{g} \,\,\int \d\P_0(\omega)\,\phi(\omega) L(g,\omega),
\end{aligned}
\end{equation}
which concludes the proof of Lemma \ref{thm-lbub-lemma1}.
\end{proof}
Now we would like to exchange $\sup_{\phi}$ and $\inf_{g}$ in \eqref{thm-lbub4}. For this, we need to invoke a {\it{compactification}} argument
based on an {\it{entropy penalization method}}. This is the the content of the next lemma. 



\begin{lemma}[Entropy penalization and coercivity in $\phi$]\label{thm-lbub-lemma2}
For every continuous and bounded function $f$ on $\Omega_0\times\mathcal U_d$,
$$
\overline H(f) \geq \liminf_{\eps\to 0}\inf_{g} \eps \log \E_0\bigg[\e^{\eps^{-1} L(g,\cdot)}\bigg]
$$
where $L(g,\cdot)$ is the functional defined in \eqref{thm-lbub5}.
\end{lemma}
\begin{proof}
We start from \eqref{thm-lbub4}. For every probability density $\phi\in L^1_+(\P_0)$, note that  its entropy functional
$$
\mathrm{Ent}(\phi)= \int \phi(\omega)\,\log\phi(\omega)\,\d\P_0(\omega).
$$
is always non-negative by Jensen's inequality. Hence, for every fixed $\eps>0$, we have a lower bound
\begin{equation}\label{thm-lbub6}
\begin{aligned}
\overline H(f) &\geq \sup_{\phi}\, \inf_{g} \bigg[\int \d\P_0(\omega)\,\phi(\omega)\bigg( L(g,\omega) - \eps \log\phi(\omega)\bigg)\bigg].
\end{aligned}
\end{equation}
Again, $\phi\mapsto \mathrm{Ent}(\phi)$ is convex and weakly lower semicontinuous in $L^1_+(\P_0)$, with its sub-level sets $\{\phi\colon \,\int \phi\log\phi\d\P_0\leq a\}$ being weakly compact in $L^1_+(\P_0)$ for all $a\in\R$. 
Also, by \eqref{thm-lbub5}, for every bounded $f$ on $\Omega_0\times\mathcal U_d$ and bounded $g$, and for every $\phi$,
$$
\int \phi(\omega) L(g,\omega)\d\P_0= \int \phi(\omega)\log\bigg(\sum_{e\in\mathcal U_d}\pi_\omega(0,e)\exp\big\{f(\omega,e)+ g(\omega)-g(\tau_e\omega)\big\}\bigg) \leq \|f\|_\infty+2 \|g\|_\infty=C<\infty.
$$
Then, if we write 
\begin{equation}\label{def-Acal}
\mathcal A_\eps(g,\phi)=\int \d\P_0(\omega)\,\big[\phi(\omega) L(g,\omega) - \eps \phi(\omega)\log\phi(\omega)\big],
\end{equation}
then, for every $\eps>0$,
as
in Lemma \ref{thm-lbub-lemma1}, for every bounded and measurable $g$, the map $g\mapsto \mathcal A_\eps(g,\phi)$ is convex and continuous and the map $\phi\mapsto \mathcal A_\eps(g,\phi)$ is concave and upper semicontinuous with compact super-level sets (i.e. the set $\{\phi\colon\, \mathcal A_\eps(g,\phi) \geq a\}$ is weakly compact for all $a\in\R$).
Applying  Von-Neumann's min-max theorem once more, we can swap the order of $\sup_\phi$ and $\inf_g$ in \eqref{thm-lbub6}. Hence,
\begin{equation}\label{thm-lbub8}
\begin{aligned}
\overline H(f) \geq \inf_{g} \,\sup_{\phi}\, \mathcal A_\eps(g,\phi)
&=\inf_{g} \eps \log \E_0\bigg[\e^{\eps^{-1} L(g,\cdot)}\bigg] \\
&\geq \liminf_{\eps\to 0}\inf_{g} \eps \log \E_0\bigg[\e^{\eps^{-1} L(g,\cdot)}\bigg].
\end{aligned}
\end{equation}
We remark that the second identity above follows from a standard perturbation argument in $\phi$ and the definition of $\mathcal A_\eps$ set in \eqref{def-Acal}. Indeed, for every admissible class of test functions $\psi$, we need to solve for $\phi$ by setting
$$
\frac{\d}{\d\eta}\bigg|_{\eta=0} \bigg[\mathcal A_\eps(g,\phi+\eta\psi)\bigg]=0
$$
for every fixed $\eps>0$ and $g$, and subject to the condition $\int \phi \d\P_0=1$. The solution is given by
$$
\phi(\cdot)= \frac{\exp\{\eps^{-1} L(g,\cdot)\}}{\E_0[\exp\{\eps^{-1} L(g,\cdot)\}]}.
$$
If we substitute this value of $\phi$ in \eqref{def-Acal}, then we are led to the identity \eqref{thm-lbub8}. This concludes the proof of Lemma \ref{thm-lbub-lemma2}.
\end{proof}
We need the following important lemma, whose proof is deferred 
until the end of the proof of Theorem \ref{thm-lbub}. Recall that $\mathcal U_d=\{\pm u_i\}_{i=1}^d$ the nearest neighbors of the origin $0$.
\begin{lemma}\label{lemma-last}
For every given $\eta>0$, there exists a sequence $\eps_n\to 0$ and a sequence $g_n$ of bounded measurable functions such that, 
\begin{equation}\label{thm-lbub9}
\eta+ \overline H(f) \geq \eps_n \log \E_0\bigg[\e^{\eps_n^{-1} L(g_n,\cdot)}\bigg],
\end{equation}
and for every $u\in\mathcal U_d$ for every $p\geq 1$,
\begin{equation}\label{def-grad-n}
G_n(\omega,u)= \1\{0\in\mathcal C_\infty\}\,\1\{\omega(u)=1\} \,\,\big(g_n(\omega)- g_n(\tau_u\omega)\big)
\end{equation}
converges weakly in $L^p(\P_0)$ as well as in distribution (as random variables) along some subsequence to some $G(\cdot,u)$. Furthermore, $G \in \mathcal G_\infty$.
\end{lemma}
We first assume the above lemma and prove Theorem \ref{thm-lbub}. 
For this purpose, we need another lemma.
\begin{lemma}\label{lemma-monotonicity}
For every $\lambda>0$, every probability measure $\mu$ and for  every random variable $X$ with finite exponential moment, if we set
$$
\psi(\lambda)=\log\E^{\ssup\mu}\big[\e^{\lambda X}\big],
$$ 
then the map $\lambda\mapsto \frac{\psi(\lambda)}{\lambda}$ is increasing in $[0,\infty)$. 
\end{lemma}
\begin{proof}
Indeed, $\psi(\lambda)$ is convex and twice differentiable in $\lambda$. In particular, 
$\psi^{\prime\prime}(\lambda)>0$, $\psi(0)=0$ and 
$$
\bigg(\frac{\psi(\lambda)}{\lambda}\bigg)^{\prime}= \frac{\psi^\prime(\lambda)}{\lambda}- \frac{\psi^\prime(\lambda)}{\lambda^2}=\frac{\lambda\psi^\prime(\lambda)-\psi(\lambda)}{\lambda^2}.
$$
Since $\lambda\psi^\prime(\lambda)-\psi(\lambda)$ is $0$ at $\lambda=0$ and $(\lambda\psi^\prime(\lambda)-\psi(\lambda))^\prime=\lambda\psi^{\prime\prime}>0$, we conclude that 
$\lambda\mapsto \frac{\psi(\lambda)}{\lambda}$ is increasing in $\lambda>0$. 
\end{proof}
We now continue with the proof of Theorem \ref{thm-lbub} assuming Lemma \ref{lemma-last}.

\noindent {\bf{Proof of Theorem \ref{thm-lbub}:}}  Fix $\eta>0$. Note that Lemma \ref{lemma-monotonicity} and \eqref{thm-lbub9} imply that for $\lambda>0$ and large enough $n$, 
$$
\begin{aligned}
\eta+ \overline H(f) &\geq \frac 1\lambda \log \E_0\big[\e^{\lambda L(g_n,\cdot)}\big].
\end{aligned}
$$
For every $M>0$, let us remark that $x\mapsto \exp\big\{\mathrm{min}\{x,M\}\big\}$ is bounded and continuous. Then we can plug in the expression for $L(g_n,\cdot)$ from \eqref{thm-lbub5} in the last bound and recall the definition of $G_n$ from \eqref{def-grad-n} to get, 
$$
\begin{aligned}
\eta+ \overline H(f) &\geq\frac 1\lambda \log \E_0\bigg[\exp\bigg\{\lambda\log\bigg(\sum_{e\in\mathcal U_d}\pi_\omega(0,e)\exp\big\{f(\omega,e)+ \mathrm{min}\{G_n(\omega,e),M\}\big\}\bigg)\bigg\}\bigg]
\end{aligned}
$$
Let us also remark that $f$ is continuous and bounded in $\Omega_0\times\mathcal U_d$. If we now let $n\to\infty$, the first part of Lemma \ref{lemma-last} implies that $G_n(\omega,e)$ converges weakly to some $G(\omega,e)$. Hence,  
\begin{equation}\label{thm-lbub10}
\eta+ \overline H(f) \geq\frac 1\lambda \log \E_0\bigg[\exp\bigg\{\lambda\log\bigg(\sum_{e\in\mathcal U_d}\pi_\omega(0,e)\exp\big\{f(\omega,e)+ \mathrm{min}\{G(\omega,e),M\}\big\}\bigg)\bigg\}\bigg].
\end{equation}
If we now let $M\uparrow\infty$ and use the monotone convergence theorem, we get 
$$
\eta+ \overline H(f) \geq\frac 1\lambda \log \E_0\bigg[\exp\bigg\{\lambda\log\bigg(\sum_{e\in\mathcal U_d}\pi_\omega(0,e)\exp\big\{f(\omega,e)+ G(\omega,e)\big\}\bigg)\bigg\}\bigg].
$$
If we now let $\lambda\to\infty$, we deduce that 
\begin{equation}\label{thm-lbub11}
\begin{aligned}
\eta+ \overline H(f) &\geq\mathrm{ess}\,\sup_{\P_0}\, \log\bigg(\sum_{e\in\mathcal U_d}\pi_\omega(0,e)\exp\big\{f(\omega,e)+ G(\omega,e)\big\}\bigg) \\
&\geq\inf_{G\in\mathcal G_\infty}\,\mathrm{ess}\,\sup_{\P_0}\, \log\bigg(\sum_{e\in\mathcal U_d}\pi_\omega(0,e)\exp\big\{f(\omega,e)+ G(\omega,e)\big\}\bigg),
\end{aligned}
\end{equation}
and in the last lower bound we invoked the second part of Lemma \ref{lemma-last} which asserts that $G\in \mathcal G_\infty$. Since the choice of $\eta>0$ was arbitrary, the last bound proves Theorem \ref{thm-lbub}, assuming Lemma \ref{lemma-last}.
\qed

We now the owe the reader the proof of Lemma \ref{lemma-last}. 

\noindent {\bf{Proof of Lemma \ref{lemma-last}:}} We will prove Lemma \ref{lemma-last} in two main steps. In the first step we will show that the sequence of formal gradients $G_n$ defined in \eqref{def-grad-n}
is uniformly integrable and converges along a subsequence to some $G$. In the next step we will show that the limit $G$ belongs to the class $\mathcal G_\infty$ introduced in Section \ref{subsec-classG}.

\noindent{\bf{Step 1: Proving $L^p(\P_0)$ boundedness and uniform integrability of $G_n$.}} First we want to prove that $G_n$ defined in \eqref{def-grad-n} is uniformly bounded in $L^p(\P_0)$ for every $p\geq 1$ and $G_n$ is also uniformly integrable. Note that by \eqref{thm-lbub8}, there exists a $\eps_n\to 0$ and a sequence $(g_n)_n$ of bounded measurable functions so that 
$$
\eps_n\log\E_0\bigg[\exp\bigg\{\eps_n^{-1} \log\bigg(\sum_{e\in\mathcal U_d}\pi_\omega(0,e)\exp\big\{f(\omega,e)+ g_n(\omega)-g_n(\tau_e\omega)\big\}\bigg)\bigg\}\bigg] \leq \overline H(f).
$$
Since $f$ is bounded, $f(\omega,e)\geq -\|f\|_\infty$ and by Lemma \ref{lemma-monotonicity}, in particular we have
\begin{equation}\label{thm-lbub12}
\E_0\bigg[\sum_{e\in\mathcal U_d}\pi_\omega(0,e)\exp\big\{g_n(\omega)-g_n(\tau_e\omega)\big\}\bigg] \leq \exp\{\overline H(f)+\|f||_\infty\}
\end{equation}
Recall that $\mathcal U_d=\{\pm u_i\}_{i=1}^d$
are the nearest neighbors of the origin $0$. For every $u=\pm u_i$, let $\Omega_{0,u}$ denote the set of configurations $\omega$
such that both $0$ and $u$ are in the infinite cluster $\mathcal C_\infty(\omega)$ and the edge $0\leftrightarrow u$ is present (i.e., $\omega_u=1$). Then $\P(\Omega_{0,u})>0$ and we set $\P_{0,u}(\cdot)=\P(\cdot| \Omega_{0,u})$. 

Now for every $u=\pm u_i$, if the edge $0\leftrightarrow u$ is present, then $\pi_\omega(0,u)\geq 1/(2d)>0$ and for some constant $C>0$, \eqref{thm-lbub12} implies
\begin{equation}\label{thm-lbub13}
\E_{0,u}\big[\exp\big\{g_n(\omega)-g_n(\tau_u\omega)\big\}\big] \leq C.
\end{equation}
Now again by \eqref{thm-lbub12}, 



$$
\begin{aligned}
\E_{0,u}\bigg[\sum_{e\in\mathcal U_d}\pi_{\tau_u\omega}(0,e)\exp\big\{g_n(\tau_u\omega)-g_n(\tau_e\tau_u\omega)\big\}\bigg] 
&= \E_0\bigg[\sum_{e\in\mathcal U_d}\pi_\omega(0,e)\exp\big\{g_n(\omega)-g_n(\tau_e\omega)\big\} \big| \omega(0,-u)=1 \bigg] \\
&\leq C \E_0\bigg[\sum_{e\in\mathcal U_d}\pi_\omega(0,e)\exp\big\{g_n(\omega)-g_n(\tau_e\omega)\big\} \bigg]\\
&\leq C\exp\{\overline H(f)+\|f||_\infty\}.
\end{aligned}
$$

Now if the edge $0\leftrightarrow u$ is present in the configuration $\omega$ (i.e., $\omega_u=1$), the edge $-u\leftrightarrow 0$ is present in the configuration $\tau_u\omega$ (i.e., $\pi_{\tau_u\omega}(0,-u)\geq 1/(2d)>0$) and hence, again
\begin{equation}\label{thm-lbub14}
\E_{0,u}\big[\exp\big\{g_n(\tau_u\omega)-g_n(\tau_{-u}\tau_u\omega)\big\}\big] =\E_{0,u}\big[\exp\big\{g_n(\tau_u\omega)-g_n(\omega)\big\}\big]\leq C.
\end{equation}
Let $G_n$ be the sequence defined in \eqref{def-grad-n}, while $G_n^+$ and $G_n^-$ denote its positive and the negative parts respectively. Then $|G_n|=G_n^++G_n^-$ and it follows from \eqref{thm-lbub13} and \eqref{thm-lbub14} that the sequence $G_n$  is uniformly bounded in $L^p(\P_0)$ for every $p\geq 1$, and thus it is also uniformly integrable and uniformly tight under $\P_0$. Consequently, $G_n$ converges weakly in $L^p(\P_0)$ and also in distribuition (as random variables) along a subsequence to some $G$. 

\noindent{\bf{Step 2: Proving that $G\in \mathcal G_\infty$.}} To conclude that $G\in \mathcal G_\infty$, note that the fact that $G$ is bounded in the essential supremum norm in $\P_0$ follows from 
the first inequality in the display \eqref{thm-lbub11}$^{1}.$\footnotetext{$^1$ Note that the display \eqref{thm-lbub11} followed only from the first part of Lemma \ref{lemma-last} (i.e., the fact that that $G_n$ converges weakly along a subsequence to some $G$), which we have just proved in Step 1. In particular, \eqref{thm-lbub11} does not use the second part of Lemma \ref{lemma-last} which asserts that $G\in \mathcal G_\infty$, which we are proving currently in Step 2.} Furthermore, the fact that $G$ satisfies the closed loop property \eqref{closedloop} on the infinite cluster $\mathcal C_\infty$ also follows easily from 
the structure of $G_n$ which is a gradient field on the infinite cluster $\mathcal C_\infty$. Indeed, 
let $0 = x_0, x_1, \dots, x_j = 0$ be a closed path in the lattice and 
let us set
$$
B(x_0,\dots,x_j)= \{ 0 = x_0, x_1, \dots, x_j=0 \textup{ is a closed path in } \mathcal C_{\infty} \}
$$
and fix an arbitrarily measurable event $A$ in $\Omega_0$. Since $\P_0=\P(\cdot|\{0\leftrightarrow\infty\})$ and $\P$ is invariant under the shifts $\tau_x$ for every $x\in\Z^d$, it follows from the weak $L^2(\P_0)$ convergence of 
$G_n$ to $G$ from Lemma \ref{lemma-last}  that for each $1 \le i \le j$,   
$$
 \lim_{n \to \infty} \E_0 \bigg[\1_{A \cap B(x_0, x_1, \dots, x_j)} (\omega) G_n(\tau_{x_{i-1}}\omega, x_i - x_{i-1})\bigg] 
= \E_0 \bigg[\1_{A \cap B(x_0, x_1, \dots, x_j)} (\omega) G(\tau_{x_{i-1}}\omega, x_i - x_{i-1})\bigg] 
$$


By the definition of $G_n$ set forth in \eqref{def-grad-n},
$$
\sum_{i=1}^j G_n(\tau_{x_{i-1}}\omega, x_i - x_{i-1}) = 0, \, \mbox{ for }  \,\P_0\textup{-a.e. } \omega \in P(x_0, x_1, \dots, x_j). 
 $$
Hence, 
$$
\E_0\bigg[ \1_{A \cap P(x_0, x_1, \dots, x_j)} (\omega) \sum_{i = 1}^{j} G(\tau_{x_{i-1}}\omega, x_i - x_{i-1})\bigg] = 0. 
$$
Since $A$ is taken arbitrarily, 
$$
\1_{B(x_0, x_1, \dots, x_j)} (\omega) \sum_{i = 1}^{j} G(\tau_{x_{i-1}}\omega, x_i - x_{i-1}) = 0, \ \P_0\textup{-a.s. }\omega. 
$$
This proves the closed loop property of $G$.


It remains to check the induced zero mean property \eqref{meanzero} of $G$. The following lemma will finish the proof of Lemma \ref{lemma-last}. Hence, the proof of
Theorem \ref{thm-lbub} will also be concluded. 

\begin{lemma}[Induced mean zero property of the limit $G$]\label{lemma-last-last}
The limiting gradient $G$ appearing in Lemma \ref{lemma-last} satisfies the induced mean zero property defined in \eqref{meanzero}. Hence, $G\in \mathcal G_\infty$.
\end{lemma}

\begin{proof}
Let us fix $e\in\mathcal U_d$
and recall that $\ell$ denotes the graph distance from $0$ to $v_e=k(\omega,e)e$, and fix $(x_0=0,x_1,\ldots,x_\ell)$ a shortest open path to $v_e=k(\omega,e)e$. 
We also recall from Section \ref{sec-classG} that the induced shift $\sigma_e\colon \Omega_0\rightarrow\Omega_0$ defined by $\sigma_e(\omega)=\tau_{k(\omega,e)e}(\omega)$ is $\P_0$-measure preserving. Hence, for every bounded measurable $g_n$
\begin{equation}\label{eq:gdifzero}
\E_0\big[g_n(\tau_{k(\omega,e)e}\omega)-g_n(\omega)\big]=
\E_0\big[g_n\circ\sigma_e-g_n\big]=0.
\end{equation}
Let us write,
$$
F_M=\E_0\bigg[V(\omega,\,k(\omega,e)e)\, \1_{\ell<M}\bigg]=
 \E_0\bigg[\sum_{j=0}^{\ell-1}\, G\big(\tau_{x_i}\omega, x_{i+1}-x_i\big)\,\, \1_{\ell<M}\bigg].
$$
We claim that, 
\begin{equation}\label{claim_FM}
F_M\to 0 \quad\mbox { as }\,\,M\to\infty.
\end{equation}
Note that, by \eqref{def-grad-n}, 
$$
F_M=\E_0\bigg[\1_{\ell\leq M}\,\sum_{j=0}^{\ell-1} \lim_{n\to\infty}\, G_n(\tau_{x_j}\omega, x_{j+1}-x_j)\bigg]
$$
We would like show that, indeed,
\begin{equation}\label{claim-FM-0}
F_M
=\E_0\bigg[\1_{\ell<M}\,\,\lim_{n\to\infty}\big(g_n(\tau_{k(\omega,e)e}\omega)-g_n(\omega)\big)\bigg],
\end{equation}
and consequently use \eqref{eq:gdifzero} to conclude
\begin{equation}\label{claim-FM-1}
\big|F_M\big|\leq\liminf_{n\to\infty}\bigg|\E_0\bigg[
{\1}_{\ell\geq M} \big(g_n(\tau_{k(\omega,e)e}\omega)-g_n(\omega)\big)\bigg]
\bigg|.
\end{equation}
Indeed, let us fix an integer $j < M$ and a finite path $0 = x_0, x_1, \dots, x_j$  from the origin. 
Let 
$$
B(x_0, x_1, \dots, x_j ) := \{\ell = j, 0 = x_0, x_1, \dots, x_j \textup{ is a path in } \mathcal C_{\infty} \}
$$
Then, 
$$
 \{\ell = j\} = \bigcup_{0 = x_0, x_1, \dots, x_j}  \{\ell = j\} \cap B(x_0, x_1, \dots, x_j ).  
 $$
We can choose $\widetilde B(x_0, x_1, \dots, x_j ) \subset B(x_0, x_1, \dots, x_j )$ such that 
we have a finite and disjoint union
\begin{equation}\label{eq-tilde-B}
\{\ell = j\} = \bigcup_{0 = x_0, x_1, \dots, x_j}  \{\ell = j\} \cap \widetilde B(x_0, x_1, \dots, x_j ).  
\end{equation}
Then it suffices to show that for each fixed path $0 = x_0, x_1, \dots, x_j$, 
$$
\lim_{n \to \infty}  
\E_0\bigg[\1_{\{\ell = j\} \cap \widetilde B(x_0, x_1, \dots, x_j )}\,\,\big(g_n(\tau_{x_j}\omega)-g_n(\omega)\big)\bigg] 
= \E_0\bigg[\1_{\{\ell = j\} \cap \widetilde B(x_0, x_1, \dots, x_j )}\,\,\sum_{i = 1}^{j} G (\tau_{x_{i-1}}\omega, x_i - x_{i-1})\bigg],
$$
Equivalently, we need to show that for each $i=1,\dots,j$, 
$$
\lim_{n \to \infty}  
\E_0\bigg[\1_{\{\ell = j\} \cap \widetilde B(x_0, x_1, \dots, x_j )}\,\,\big(g_n(\tau_{x_i}\omega) - g_n(\tau_{x_{i-1}}\omega)\big)\bigg] 
= \E_0\bigg[\1_{\{\ell = j\} \cap \widetilde B(x_0, x_1, \dots, x_j )}\,\,G (\tau_{x_{i-1}}\omega, x_i - x_{i-1})\bigg],
$$
But this statement follows from the fact that
$$
\1_{\{\ell = j\} \cap \widetilde B(x_0, x_1, \dots, x_j )}\,\,\big(g_n(\tau_{x_i}\omega) - g_n(\tau_{x_{i-1}}\omega)\big) = \1_{\{\ell = j\} \cap \widetilde B(x_0, x_1, \dots, x_j )}
G_n (\tau_{x_{i-1}}\omega, x_i - x_{i-1}). 
$$
and Lemma \ref{lemma-last} which implies convergence of $G_n (\tau_{x_{i-1}}\omega, x_i - x_{i-1})$  to $G (\tau_{x_{i-1}}\omega, x_i - x_{i-1})$ along a subsequence weakly in $L^2$ under the measure $\P(\cdot | x_{i-1}, x_i  \in \mathcal{C}_{\infty})$. This proves \eqref{claim-FM-0}, which also implies \eqref{claim-FM-1} when combined with \eqref{eq:gdifzero}.

We now need to estimate the right hand side of \eqref{claim-FM-1}, which is  
\begin{equation}\label{eq-claim-FM-3}
\begin{aligned}
\bigg|\E_0\bigg[
{\1}_{\ell\geq M} \bigg(g_n(\tau_{k(\omega,e)e}\omega)-g_n(\omega)\bigg)\bigg] \bigg|
&=\bigg|\E_0\bigg[
\sum_{j=M}^\infty{\1}_{\ell= j} \bigg(\sum_{i=0}^{j-1} \big(g_n(\tau_{x_i}\omega)-g_n(\tau_{x_{i+1}}\omega)\big)\bigg)\bigg] \bigg|\\
&\leq\E_0\bigg[
\sum_{j=M}^\infty{\1}_{\ell= j} \bigg(\sum_{i=0}^{j-1} \big|G_n(\tau_{x_i}\omega,x_{i+1}-x_i)\big|\bigg)\bigg] \bigg| \\
\end{aligned}
\end{equation}
For each $j\geq M$ we will now estimate $\E_0\bigg[{\1}_{\ell= j} \bigg(\sum_{i=0}^{j-1} |G_n(\tau_{x_i}\omega, x_{i+1}-x_i)|\bigg)\bigg]$. 
For this, it is enough to estimate 
$$
\sum_{0 = x_0, \dots, x_j} \E_0\bigg[{\1}_{\ell= j \cap \widetilde P (x_0, \dots, x_j)} \bigg(\sum_{i=0}^{j-1} G_n(\tau_{x_i}\omega, x_{i+1}-x_i)\bigg)\bigg]
$$
where $\widetilde B(x_0,\dots,x_j)$  is the event defined before so that \eqref{eq-tilde-B} holds. 

Then by Lemma \ref{lemma-ell} , 
we can choose two constants $c_1, c_2$ such that for every $j \ge 1$, 
$$
\P_0 (\ell = j) \le c_1 \exp(- c_2 j). 
 $$
Furthermore, let $p, q \geq 1$ be such that $(1/p) + (1/q) =1$ and 
\begin{equation}\label{eq-rho}
\rho:=\frac{c_1 (2d)^{1/q}} {\exp(c_2 / p)} <1. 
\end{equation}
Now by the H\"older's inequality, writing $|G_n|=G_n^+ +  G_n^-$ and invoking \eqref{thm-lbub13} and \eqref{thm-lbub14} again, we get
$$
\begin{aligned}
& \E_0\bigg[{\1}_{\{\ell= j\} \cap \widetilde B (x_0, \dots, x_j)} |G_n(\tau_{x_i}\omega, x_{i+1}-x_i)| \bigg] 
\\
&\leq \E_0\bigg[ {\1}_{x_{i}, x_{i+1} \in \mathcal{C}_{\infty}} |G_n(\tau_{x_i}\omega, x_{i+1}-x_i)|^q \bigg]^{1/q}  \P_0 \bigg[\{\ell= j\} \cap \widetilde B (x_0, \dots, x_j))\bigg]^{1/p}. 
\\
&\leq C_q \P_0 \bigg[\{\ell= j\} \cap \widetilde B (x_0, \dots, x_j)\bigg]^{1/p}. 
\end{aligned}
$$

Since the number of paths $0 = x_0, \dots, x_j$ is bounded by $(2d)^j$, by the H\"older's inequality, 
$$
\sum_{0 = x_0, \dots, x_j} \P_0 (\ell= j \cap \widetilde P (x_0, \dots, x_j))^{1/p} \le \left(c_1 (2d)^{1/q}\exp(-c_2 / p)\right)^j.
$$

Hence,
$$
\begin{aligned}
 \sum_{0 = x_0, \dots, x_j} \E_0\bigg[{\1}_{\ell= j \cap \widetilde P (x_0, \dots, x_j)} \bigg(\sum_{i=0}^{j-1} G_n(\tau_{x_i}\omega, x_{i+1}-x_i)\bigg)\bigg] 
 &\leq j \left(c_1 (2d)^{1/q}\exp(-c_2 / p)\right)^j 
 \\
 &=j \rho^j
 \end{aligned}
 $$
Since $\rho<1$ by our assumption in \eqref{eq-rho}, the claim \eqref{claim_FM} follows from \eqref{eq-claim-FM-3} and \eqref{claim-FM-1}.


Recall the definition of the corrector $V(\omega,k(\omega,e)e)=\sum_{i=0}^{\ell-1} G(\tau_{x_i}\omega,x_{i+1}-x_i)$ corresponding to the limit $G$ of $G_n$. To prove the induced mean zero property \eqref{meanzero} 
for $V$, we have to show that $\E_0\big[V(\omega,k(\omega,e)e)\big]=0$. For this, we note that $V(\omega,k(\omega,e)e)$ is the 
almost sure pointwise limit of $V(\omega,k(\omega,e)e){\1}_{\ell\leq M}$ as $M\to \infty$.
Furthermore, 
$|V(\omega,k(\omega,e)e){\1}_{\ell\leq M}|\leq \|G\|_\infty\ell$ for all $M$ and $\E_0 (\ell) <\infty$ by 
Lemma \ref{lemma-ell},
so by the dominated convergence theorem
$$
\E_0\big[V(\omega,k(\omega,e)e)\big] = \lim_{M\to\infty}\E_0\big[V(\omega,k(\omega,e)e)\,{\1}_{\ell\leq M}\big] = \lim_{M\to\infty}F_M= 0,
$$
while the last identity follows from \eqref{claim_FM}. We conclude that $G$ satisfies the induced mean zero property defined in \eqref{meanzero}. Hence, $G\in \mathcal G_\infty$ and
the proofs of Lemma \ref{lemma-last-last} and that of Lemma \ref{lemma-last} are finished. This also concludes the proof of Theorem \ref{thm-lbub}.
\end{proof}

\begin{remark}\label{rmk-simplify}
As mentioned in Section \ref{sec-results-3}, let us point out that in the case of an elliptic RWRE, our arguments based on entropic coercivity used in Lemma \ref{thm-lbub-lemma1} and Lemma \ref{thm-lbub-lemma2}
simplify the earlier approach used in \cite{R06,Y08}. Indeed, if $\P$ continues to denote the law of a stationary and ergodic random environment with $\pi$ being the transition  
probabilities of a random walk in $\Z^d$ such that $\int |\log\pi|^p \,\d\P<\infty$ with $p>d$,  then following our arguments before, we can set (as in \eqref{def-grad-n}) $G_n(\omega,e)=g_n(\omega)-g_n(\tau_e\omega)$.
Then following Step 1 of the proof of Lemma \ref{lemma-last} we can show that $\{G_n^+\}_n$ as well as $\{G_n^-\}_n$, and hence $\{G_n\}_n$, all remain bounded in $L^p(\P)$.
Then every weak limit $G$ of $G_n$ defined here is immediately a gradient so that the closed loop as well as the mean zero property (i.e. $\E^\P(G)=0$) is readily satisfied, thanks to translation invariance of $\P$. 
\end{remark}

\section{Proofs of Theorem \ref{thmmomgen}, Theorem \ref{thmlevel2}, Corollary \ref{thmlevel1} and Lemma \ref{nonlsc}}\label{sec-6}

\noindent {\bf{Proof of Theorem \ref{thmmomgen}:}} The proof of Theorem \ref{thmmomgen} is readily finished by combining the lower bound from Corollary \ref{corlb}, the upper bound from Lemma \ref{ub} and the equivalence of bounds 
from Theorem \ref{thm-lbub}.
\qed

\noindent {\bf{Proof of Theorem \ref{thmlevel2}:}} By Theorem \ref{thmmomgen}, 
$$
\lim_{n\to\infty} \frac 1n \log E^{\pi,\omega}_0 \big\{\exp\big\{n \langle f, \mathfrak L_n\rangle\big\}\big\} = \sup_{\mu\in\Mcal_1^\star} \big\{ \langle f, \mu\rangle- \mathfrak I(\mu)\big\}= \sup_{\mu\in\Mcal_1(\Omega_0 \times  \mathcal U_d)} \big\{ \langle f, \mu\rangle- \mathfrak I(\mu)\big\}= \mathfrak I^\star(\mu).
$$
Since $\Omega_0$ is a closed subset of $\Omega=\{0,1\}^{\mathbb B_d}$ and hence, is compact, $\Mcal_1(\Omega_0 \times  \mathcal U_d)$ is compact in the weak topology. The upper bound \eqref{ldpub} for all closed sets now follows from
Theorem 4.5.3 \cite{DZ98}. The lower bound \eqref{ldplb} has been proven in Lemma \ref{lemmalb}.
\qed

\noindent {\bf{Proof of Corollary \ref{thmlevel1}}}: The claim follows by contraction principle 
once we show that $\inf_{\xi(\mu)=x} \mathfrak I(\mu)= \inf_{\xi(\mu)=x} \mathfrak I^{\star\star}(\mu)$. This is easy to check using
convexity of $\mathfrak I$ and $\mathfrak I^{\star\star}$.
\qed
\vspace{5mm}

\noindent {\bf{Proof of Lemma \ref{nonlsc}: The zero speed regime of SRWPC under a drift.}}

\noindent For every $\beta>1$, we define
$$
\pi^{\ssup \beta}(\omega,e)= \frac{V(e) \1_{\{\omega(e)=1\}}} {\sum_{e^\prime\in \mathcal U_d}V(e^\prime) \1_{\{\omega(e^\prime)=1\}}} \in \widetilde\Pi,
$$
where
$$
V(e)=
\begin{cases}
\beta>1 \qquad\mbox{if } e=e_1,\\
1 \qquad\qquad\mbox{else}.
\end{cases}
$$
Let $X_n^{\ssup \beta}$ be the Markov chain with transition probabilities $\pi^{\ssup\beta}$. By \cite{BGP02} and \cite{Sz02}, there exists $\beta_u=\beta_u(p,d)>0$ so that for $\beta>\beta_u$, the limiting speed 
$$
\lim_{n\to\infty} \frac{X_n^{\ssup \beta}}{n},
$$
which exists and is an almost sure constant is zero. For the Bernoulli (bond and site) percolation cases the last statement follows from \cite{BGP02} and \cite[Theorem 4.1]{Sz02}. Since the finite energy property (recall the proof of Lemma \ref{lemma-ell}  for the random cluster model) holds for the random-cluster model \cite{G06} and level sets of Gaussian free field \cite[Remark 1.6]{RS13}, the proof of \cite[Theorem 4.1]{Sz02} is applicable, while
in the case of random interlacements (for which the finite energy property fails), the statement regarding the zero speed of the random walk $X_n^{\ssup \beta}$ follows from \cite{FP16}.

\noindent Then, by Kesten's lemma (see \cite{k75}),
there exists no $\phi\in L^1(\P_0)$ so that $(\pi^{\ssup\beta},\phi)\in \mathcal E$. We split the proof into two cases.

Suppose there exists a neighborhood $\mathfrak u$ of $\pi^{\ssup\beta}$ so that every $\tilde\pi^{\ssup\beta} \in \overline {\mathfrak u}$ fails to have an invariant density. Then, 
for every $\tilde\pi^{\ssup\beta} \in \overline {\mathfrak u}$ and any probability density $\phi\in L^1(\P_0)$, let $\mu_\beta$ be the corresponding element in $\Mcal_1(\Omega_0\times\mathcal U_d)$ (i.e., 
$\d\mu_\beta(\omega,e)=\pi^{\ssup\beta}(\omega,e)\phi(\omega)\d\P_0(\omega)$). Since 
$$(\tilde\pi^{\ssup\beta},\phi)\notin\mathcal E,
$$
by Lemma \ref{onetoone}, $\mu_\beta\notin   \Mcal_1^\star$. Then, $\mathfrak I(\mu_\beta)=\infty$ by \eqref{Idef}. If $\mathfrak I$ were lower semicontinuous on $\Mcal_1(\Omega_0\times \mathcal U_d)$, then $\mathfrak I=\mathfrak I^{\star\star}$ and by Theorem \ref{thmlevel2},
\begin{equation}\label{eq:superexp}
P^{\pi,\omega}_0\big\{\mathfrak L_n \in \mathfrak n\big\}
\end{equation}
would decay super-exponentially for $\P_0$- almost every $\omega\in \Omega_0$, with $\mathfrak n$ being some neighborhood of $\mu_\beta$. However, since for every $\omega$, the relative entropy of $\pi^{\ssup\beta}(\omega,\cdot)$ w.r.t. $\pi_\omega(0,\cdot)$ is bounded below and above, the probability in \eqref{eq:superexp} decays exponentially and we have a contradiction.

Assume that there exists no such neighborhood $\mathfrak u$ of $\pi^{\ssup\beta}$. Let $\tilde\pi_n\to\pi^{\ssup\beta}$ such that for all $n\in \N$, $\tilde\pi_n$ has an invariant density $\phi_n$ and $(\tilde\pi_n,\phi_n)\in\mathcal E$. If $(\mu_n)_n$ is the sequence corresponding to $(\tilde\pi_n,\phi_n)$, since $\Mcal_1(\Omega_0\times\mathcal U_d)$ is compact, $\mu_n\Rightarrow\mu_\beta$ weakly along a subsequence. However, by our choice of $\beta>\beta_u$, $(\pi^{\ssup\beta},\phi)\notin \mathcal E$ for every density $\phi$ and hence $\mu_\beta\notin \Mcal_1^\star$ and $\mathfrak I(\mu_\beta)=\infty$. But,
$$
\lim_{n\to\infty} \mathfrak I(\mu_n)= \int \d\P_0 \phi(\omega) \sum_{e\in \mathcal U_d} \pi^{\ssup\beta}(\omega,e) \log \frac{\pi^{\ssup\beta}(\omega,e)}{\pi_\omega(0,e)},
$$
which is clearly finite. This proves that $\mathfrak I$ is not lower semicontinuous. 
\qed

{\bf{Acknowledgments.}} Most of the work in this paper was carried out when the second author was a visiting assistant professor at the Courant Institute of Mathematical Sciences, New York University
in the academic year 2015-2016, and its hospitality is gratefully acknowledged. The third author was supported by Grant-in-Aid for Research Activity Start-up (15H06311) and Grant-in-Aid for JSPS Fellows (16J04213). 
The authors would like to thank S. R. S. Varadhan for reading an early draft of the manuscript and many valuable suggestions. 
We also thank Firas Rassoul-Agha and Atilla Yilmaz for their comments on a preprint version of the present article. Finally, we would like to 
thank an anonymous referee for a very careful reading and many valuable suggestions that led to a more elaborate version of our manuscript.


\end{document}